\documentclass{amsart}

\usepackage[latin1]{inputenc}
\usepackage{amssymb,amsfonts}
\usepackage{amsmath, amsthm}
\usepackage{euscript}
\input xy \xyoption {all}
\usepackage{mathrsfs}

\RequirePackage{color}
\definecolor{myred}{rgb}{0.75,0,0}
\definecolor{mygreen}{rgb}{0,0.5,0}
\definecolor{myblue}{rgb}{0,0,0.65}

\RequirePackage{ifpdf}
\ifpdf
  \IfFileExists{pdfsync.sty}{\RequirePackage{pdfsync}}{}
  \RequirePackage[pdftex,
   colorlinks = true,
   urlcolor = myblue, 
   citecolor = mygreen, 
   linkcolor = myred, 
   bookmarksopen=true]{hyperref}
\else
  \RequirePackage[hypertex]{hyperref}
\fi

\RequirePackage{ae, aecompl, aeguill} 

\numberwithin{equation}{subsection}

  \def\bg{{\mathfrak b}}  
    \def\CM{{\mathbb{C}}}
    \def\DM{{\mathbb{D}}}
    
    \def\FM{{\mathbb{F}}}
  \def\gg{{\mathfrak g}}  
  \def\hg{{\mathfrak h}}  \def\HM{{\mathbb{H}}}
\def\IG{{\mathfrak I}}    \def\IM{{\mathbb{I}}}
\def\JG{{\mathfrak J}}    \def\JM{{\mathbb{J}}}
  \def\kg{{\mathfrak k}}  
  \def\lg{{\mathfrak l}}  
  \def\mg{{\mathfrak m}}  
  \def\ng{{\mathfrak n}}  
    
  \def\pg{{\mathfrak p}}  
  \def\qg{{\mathfrak q}}  \def\QM{{\mathbb{Q}}}
\def\RG{{\mathfrak R}}    \def\RM{{\mathbb{R}}}
\def\SG{{\mathfrak S}}  \def\sg{{\mathfrak s}}  
\def\TG{{\mathfrak T}}  \def\tg{{\mathfrak t}}  \def\TM{{\mathbb{T}}}

    \def\XM{{\mathbb{X}}}
    
    \def\ZM{{\mathbb{Z}}}

\def\BB{{\mathbf B}}    
    \def\CC{{\mathcal{C}}}
    
    \def\EC{{\mathcal{E}}}
    \def\FC{{\mathcal{F}}}
\def\GB{{\mathbf G}}    \def\GC{{\mathcal{G}}}
\def\HB{{\mathbf H}}    
\def\IB{{\mathbf I}}    \def\IC{{\mathcal{I}}}
\def\JB{{\mathbf J}}    
    
\def\LB{{\mathbf L}}    
    
    \def\NC{{\mathcal{N}}}
    \def\OC{{\mathcal{O}}}
\def\PB{{\mathbf P}}

    \def\SC{{\mathcal{S}}}
\def\TB{{\mathbf T}}    \def\TC{{\mathcal{T}}}
\def\UB{{\mathbf U}}

\def\DS{{\EuScript D}}

\newcommand{\nc}{\newcommand} \newcommand{\renc}{\renewcommand}

\def\reg{{\mathrm{reg}}}

\def\rs{{\mathrm{rs}}}

\DeclareMathOperator{\Ker}{Ker}
\DeclareMathOperator{\Coker}{Coker}

\DeclareMathOperator{\Coh}{Coh}

\def\to{\rightarrow}

\def\longto{\longrightarrow}

\nc{\triright}{\stackrel{[1]}{\to}}
\nc{\longtriright}{\stackrel{[1]}{\longto}}

\nc{\Br}{\mathcal{B}}
\nc{\HotRR}{{}_R\mathcal{K}_R}
\nc{\HotR}{\mathcal{K}_R}
\nc{\excise}[1]{}
\nc{\defect}{\text{df}}
\nc{\h}[1]{\underline{H}_{#1}}

\nc{\Ga}{\mathbb{G}_a} 
\nc{\Gm}{\mathbb{G}_{\mathrm{m}}} 

\nc{\Perv}{{\mathbf{P}}}

\nc{\IH}{{\mathrm{IH}}}

\nc{\ic}{\mathbf{IC}}

\nc{\gl}{{\mathfrak{gl}}}
\renc{\sl}{{\mathfrak{sl}}}
\renc{\sp}{{\mathfrak{sp}}}

\nc{\HBM}{H^{BM}}

 \DeclareMathOperator{\Hom}{Hom}

\DeclareMathOperator{\Rep}{Rep}

\newtheorem{thm}{Theorem}[subsection]
\newtheorem{lem}[thm]{Lemma}

\newtheorem{prop}[thm]{Proposition}
\newtheorem{cor}[thm]{Corollary}

\theoremstyle{definition}

\theoremstyle{remark}
\newtheorem{remark}[thm]{Remark}

\DeclareMathOperator{\Ext}{Ext}

\DeclareMathOperator{\Spec}{Spec}

\DeclareMathOperator{\Parity}{Parity}
\DeclareMathOperator{\Tilt}{Tilt}
\DeclareMathOperator{\PParity}{PParity}

\def\pt{{\mathrm{pt}}}

\def\Gr{{\EuScript Gr}}

\def\tNC{\widetilde{\NC}}
\def\Flag{\mathscr{B}}

\newcommand{\ggr}{\gg_{\mathrm{reg}}}
\newcommand{\tSC}{\widetilde{\SC}}
\newcommand{\tUp}{\widetilde{\Upsilon}}

\def\GD{\check{G}}
\def\BD{\check{B}}
\def\UD{\check{U}}
\def\PD{\check{P}}
\def\TD{\check{T}}
\def\LD{\check{L}}

\def\Vect{\mathrm{Vect}}

\newcommand{\tgg}{\widetilde{\gg}}

\makeatletter
\def\lotimes{\@ifnextchar_{\@lotimessub}{\@lotimesnosub}}
\def\@lotimessub_#1{\mathchoice{\mathbin{\mathop{\otimes}^L}_{#1}}%
  {\otimes^L_{#1}}{\otimes^L_{#1}}{\otimes^L_{#1}}}
\def\@lotimesnosub{\mathbin{\mathop{\otimes}^L}}
\makeatother

\newcommand{\simto}{\xrightarrow{\sim}}

\newcommand{\Db}{D^{\mathrm{b}}}

\newcommand{\tIB}{\widetilde{\IB}}
\newcommand{\tIG}{\widetilde{\IG}}
\newcommand{\tJB}{\widetilde{\JB}}
\newcommand{\tJG}{\widetilde{\JG}}
\newcommand{\tIM}{\widetilde{\IM}}
\newcommand{\tJM}{\widetilde{\JM}}

\newcommand{\Lie}{\mathscr{L} \hspace{-1pt} ie}
\newcommand{\LLie}{\mathbb{L} \mathrm{ie}}

\newcommand{\GDO}{\GD(\mathscr{O})}
\newcommand{\LDO}{\LD(\mathscr{O})}
\newcommand{\FF}{\mathsf{F}}
\newcommand{\HF}{\mathsf{H}}

\newcommand{\Satake}{\mathsf{S}}
\newcommand{\Dmix}{D^{\mathrm{mix}}}
\newcommand{\Kb}{K^{\mathrm{b}}}

\newcommand{\Rder}{\mathsf{R}}

\newcommand{\Kil}{\mathrm{Kil}}

\begin{document}

\begin{abstract}
We prove analogues of fundamental results of Kostant on the universal centralizer of a connected reductive algebraic group for algebraically closed fields of positive characteristic (with mild assumptions), and for integral coefficients. As an application, we use these results to obtain a ``mixed modular'' analogue of the derived Satake equivalence of Bezrukavnikov--Finkelberg.
\end{abstract}

\title[Kostant section, universal centralizer, derived Satake]{Kostant section, universal centralizer, \\
and a modular derived Satake equivalence}
  
\author{Simon Riche}
\address{Universit{\'e} Clermont Auvergne, Universit{\'e} Blaise Pascal, Laboratoire de Math{\'e}matiques, BP 10448, F-63000 Clermont-Ferrand, France -- CNRS, UMR 6620, LM, F-63178 Aubi{\`e}re, France}
\email{simon.riche@math.univ-bpclermont.fr}

\thanks{The author was supported by ANR Grants No.~ANR-2010-BLAN-110-02 and ANR-13-BS01-0001-01.}

\maketitle

\section{Introduction}

\subsection{Introduction}

The main goal of this note is to provide a detailed exposition of a generalization to arbitrary characteristic (and to integral coefficients) of some fundamental results of Kostant on the universal centralizer group scheme $\IB$ (and its Lie algebra $\IG$) associated with a connected reductive group $\GB$. Some of these results are sometimes considered ``well known'' but, as far as we know, are not treated in detail in the literature in this generality. Therefore, we believe they deserve a complete treatment, under assumptions as mild as possible.

In the case of fields of characteristic $0$, these results play an important technical role in various aspects of the geometric Langlands program, see e.g.~\cite{ginzburg-satake, bf, dodd, ngo}. More recently, the ``modular'' case treated in this paper has been used in a joint work with Carl Mautner~\cite{mr} (again, as a technical tool) to generalize some constructions of~\cite{dodd} to positive characteristic, with an application to the completion of the proof of the Mirkovi{\'c}--Vilonen conjecture on stalks of standard spherical perverse sheaves on affine Grassmannians~\cite{mv}.
In this paper we present another application, to a ``mixed modular'' version of the Bezrukavnikov--Finkelberg \emph{derived Satake equivalence}~\cite{bf}, an equivalence of monoidal triangulated categories relating the $\GDO$-equivariant derived category on the affine Grassmannian $\Gr$ of a reductive group $\GD$ and the derived category of equivariant coherent sheaves on the Lie algebra $\gg$ of the Langlands dual group $\GB$.

\subsection{Kostant's transversal slice}

The proofs of the results on the universal centralizer use in a crucial way a nice transversal slice to the regular nilpotent orbit in the Lie algebra $\gg$ of $\GB$. After the introduction of this slice by Kostant (in the case of complex coefficients) in~\cite{kostant}, the study of transverse slices to more general nilpotent orbits has become a subject of intense study, see e.g.~\cite{slodowy, spaltenstein, premet}. The case of the Kostant section has some special features, however. First of all, its definition is very explicit and quite elementary, and therefore well suited to a generalization to integral coefficients. Secondly, since it is concerned with the \emph{regular} nilpotent orbit, it captures some features of the whole locus of regular elements in $\gg$ (which plays a crucial role in the study of the universal centralizer) but also of the adjoint quotient $\gg/G$. These considerations do not make sense (at least in this form) for more general nilpotent orbits.

\subsection{Description of the main results}

After some preliminaries in Section~\ref{sec:preliminaries},
the first main part of the paper (Section~\ref{sec:fields}) is concerned with the case of algebraically closed fields.
In this case we have tried to prove the various results in the maximal reasonable generality, and to introduce our assumptions only when they become necessary.

First we explain the construction of the \emph{Kostant section} $\SC$, which is both a transversal slice to the regular nilpotent orbit in $\gg$ and a section of the adjoint quotient $\gg \to \gg/\GB$
(see Theorem~\ref{thm:kostant}). In the case of positive characteristic, this construction uses the detailed analysis of the adjoint action of a regular nilpotent element due to Springer~\cite{springer}. Then we consider the universal centralizer $\IB$, i.e.~the group scheme over $\gg$ whose fiber over a point $x \in \gg$ is the (scheme-theoretic) centralizer $\GB_x$ of $x$ in $\GB$. Using $\SC$, we prove that the restriction $\IB_\reg$ of $\IB$ to the regular locus $\ggr$ is a smooth commutative group scheme (see in particular Corollary~\ref{cor:Ireg-smooth}). We also describe the Lie algebra of $\IB_\reg$ in terms of the cotangent bundle to $\gg/\GB$ (see Theorem~\ref{thm:Lie-alg-cotangent}), generalizing results in characteristic $0$ due (to the best of our knowledge) to Bezrukavnikov--Finkelberg~\cite{bf}. Finally we consider variants of these results when the Lie algebra $\gg$ is replaced by the Grothendieck resolution $\tgg$ (see~\S\ref{ss:groth-resolution}).

The second main part of the paper (Section~\ref{sec:R}) is concerned with the case when $\GB$ is a split connected, simply-connected, semi-simple algebraic group over the finite localization $\RG$ of $\ZM$ obtained by inverting all the prime numbers which are not very good, or a general linear group over $\ZM$.
In this setting, using again the results of Springer we define and study a Kostant section $\SC$ (see in particular Theorem~\ref{thm:kostant-R}), prove that the restriction of the universal centralizer to $\SC$ is smooth (see Proposition~\ref{prop:IS-smooth-R}), and describe the Lie algebra of this restriction in terms of the cotangent bundle to $\SC$ (see Theorem~\ref{thm:Lie-alg-cotangent-R}). Finally we study analogues of these objects for the Grothendieck resolution (see~\S\ref{ss:grothendieck-R}).
The proofs in this section mainly proceed by reduction to the case of algebraically closed fields.
(The necessary general results related to this technique are treated in~\S\ref{ss:reduction-fields}.)

\subsection{Application to a derived Satake equivalence}
\label{ss:intro-derived-Satake}

In Section~\ref{sec:derived-Satake} we present an application of some of these results to a ``mixed modular version'' of the derived Satake equivalence of~\cite{bf}. Here we let $\GD$ be a simply-connected quasi-simple complex algebraic group, and consider the affine Grassmannian $\Gr:=\GD(\mathscr{K})/\GDO$, where $\mathscr{K}:=\CM ( \hspace{-1pt} (z) \hspace{-1pt} )$ is the field of Laurent series in an indeterminate $z$, and $\mathscr{O}:=\CM [ \hspace{-1pt} [z] \hspace{-1pt} ]$ is the ring of power series in $z$. We also let $\FM$ be an algebraically closed field whose characteristic $\ell$ is very good for $\GD$, and denote by $\GB$ the simple $\FM$-algebraic group (of adjoint type) which is Langlands dual to $\GD$.

Following earlier work with Pramod Achar~\cite{modrap2}, we define the ``mixed modular'' $\GDO$-equivariant derived category of $\Gr$ as
\[
\Dmix_{\GDO}(\Gr, \FM) := \Kb \Parity_{\GDO}(\Gr, \FM),
\]
where $\Parity_{\GDO}(\Gr, \FM)$ is the additive category of $\GDO$-equivariant parity sheaves on $\Gr$, in the sense of Juteau--Mautner--Williamson~\cite{jmw}. To justify this definition we remark that, in the case $\ell=0$, one can show using results of Rider~\cite{rider} that this category is a mixed version (in the natural sense) of the usual equivariant derived category $\Db_{\GDO}(\Gr, \FM)$. In the case $\ell>0$, a similar definition has been used in the setting of Bruhat-constructible sheaves on (finite dimensional) flag varieties in~\cite{modrap2}; in this case also we were able to show that the mixed modular derived category is a mixed version of the ordinary derived category of constructible sheaves. We expect that a generalization of the results of~\cite{modrap2} will lead to a proof of the similar property for $\Dmix_{\GDO}(\Gr, \FM)$.

The category $\Parity_{\GDO}(\Gr, \FM)$ has a natural convolution product, which induces a similar convolution product on $\Dmix_{\GDO}(\Gr, \FM)$.
In Theorem~\ref{thm:derived-Satake} we obtain an equivalence of monoidal triangulated categories
\begin{equation}
\label{eqn:intro-derived-Satake}
\Dmix_{\GDO}(\Gr, \FM) \simto \Db_{\mathrm{fr}} \Coh^{\GB \times \Gm}(\gg),
\end{equation}
where on the right-hand side $\GB \times \Gm$ acts on $\gg$ via $(g,t) \cdot x = t^{-2} (g \cdot x)$, and $\Db_{\mathrm{fr}} \Coh^{\GB \times \Gm}(\gg)$ is the triangulated subcategory of the bounded derived category $\Db \Coh^{\GB \times \Gm}(\gg)$ of $\GB \times \Gm$-equivariant coherent sheaves on $\gg$ generated by the ``free'' objects of the form $V \otimes \OC_{\gg}$, for $V$ a finite-dimensional $\GB \times \Gm$-module. This equivalence sends an object $\FC$ of $\Dmix_{\GDO}(\Gr, \FM)$ which is both parity and perverse to $\Satake(\FC) \otimes \OC_{\gg}$, where $\Satake$ is the Satake equivalence from~\cite{mv}. It is an analogue, in our ``mixed modular'' context, of one of the equivalences constructed (in the case $\ell=0$) in~\cite{bf}.

The general strategy of our proof of~\eqref{eqn:intro-derived-Satake} is similar to the one used in~\cite{bf}. However we replace some of the arguments based on explicit computations in semisimple rank $1$ by more general considerations based on results concerning the (equivariant) cohomology of $\Gr$ and of spherical perverse sheaves due, in the case of characteristic $0$, to Ginzburg~\cite{ginzburg} and, in the general case, to Yun--Zhu~\cite{yz}. 

This application only uses the results obtained in~\S\S\ref{ss:principal-nilpotent}--\ref{ss:centralizer-F}; in particular, the case of integral coefficients is not needed. The other results (and in particular the integral case) are used in~\cite{mr}.

\subsection{Some notation and conventions}

All rings in this paper are tacitly assumed to be unital and commutative. If $A$ is a $B$-algebra, we denote by $\Omega_{A/B}$ the $A$-module of relative differential forms, see~\cite[\S II.8]{hartshorne}. If $M$ is an $A$-module, we denote by $\mathrm{S}_A(M)$ (or sometimes simply $\mathrm{S}(M)$) the symmetric algebra of $M$ over $A$.

If $X$ is a scheme, we denote by $\OC_X$ its structure sheaf, and by $\OC(X)$ the global sections of $\OC_X$. 
If $Y$ is a scheme and $X$ is a $Y$-scheme, we denote by $\Omega_{X/Y}$ the $\OC_X$-module of relative differential forms. If $X$ is smooth over $Y$, we denote by $\TM^*(X/Y)$ (or simply $\TM^*(X)$) the cotangent bundle to $X$ over $Y$, in other words the relative spectrum of the symmetric algebra of the locally free $\OC_X$-module $\TC_{X/Y}:=\mathscr{H} \hspace{-1pt} om_{\OC_X}(\Omega_{X/Y}, \OC_X)$.

If $k$ is a ring and $V$ is a free $k$-module of finite rank, by abuse we still denote by $V$ the affine $k$-scheme $\Spec \bigl( \mathrm{S}_k(\Hom_k(V,k)) \bigr)$.


If $\RG$ is a finite localization of $\ZM$, we define a \emph{geometric point} of $\RG$ to be an algebraically closed field whose characteristic is not invertible in $\RG$. If $\FM$ is a such a geometric point, then there exists a unique algebra morphism $\RG \to \FM$, so that tensor products of the form $\FM \otimes_\RG (-)$ make sense.

\subsection{Acknowledgements}

This work is part of a joint project with Carl Mautner, see~\cite{mr-exotic, mr}, and was motivated by discussions with him.
We also thank Zhiwei Yun for useful conversations on regular elements and for confirming a sign mistake in~\cite{yz}, and Sergey Lysenko for his help with some references. Finally, we thank an anonymous referee for his criticism, which lead us to work on the application presented in Section~\ref{sec:derived-Satake}.

\section{Preliminaries on Lie algebras and regular elements}
\label{sec:preliminaries}

In this section we recall a number of definitions and basic results on Lie algebras of group schemes and regular elements in the Lie algebra of a reductive group over an algebraically closed field.

\subsection{Lie algebra of group schemes}
\label{ss:Lie}

Let $X$ be a Noetherian scheme, and let $G$ be a Noetherian group scheme over $X$. We denote by $\varepsilon \colon X \to G$ the identity, and consider the $\OC_X$-coherent sheaf
\[
\omega_{G/X} := \varepsilon^*(\Omega_{G/X}).
\]
Then, by definition (see e.g.~\cite[Expos{\'e} II, \S 4.11]{sga}; see also~\cite[Chap.~12]{waterhouse} for the case $G$ and $X$ are affine), the Lie algebra of $G$ is the quasi-coherent $\OC_X$-module
\[
\Lie(G/X) := \mathscr{H} \hspace{-1pt} om_{\OC_X}(\omega_{G/X}, \OC_X),
\]
with its natural bracket. We also denote by $\LLie(G/X)$ the scheme over $X$ which is the relative spectrum of the symmetric algebra of the $\OC_X$-module $\omega_{G/X}$. This scheme is naturally a Lie algebra over $X$. When $X$ is clear, we will sometimes abbreviate $\Lie(G/X)$, resp.~$\LLie(G/X)$, to $\Lie(G)$, resp.~$\LLie(G)$.

We will use the following easy consequences of the definition.

\begin{lem}
\label{lem:Lie-alg}
\begin{enumerate}
\item
If $G$ is smooth over $X$, then the $\OC_X$-module $\Lie(G/X)$ is locally free of finite rank, and $\LLie(G/X)$ is a vector bundle over $X$.
\label{it:Lie-smooth}
\item
Let $f \colon Y \to X$ be a morphism. Assume that either $G$ is smooth over $X$, or $f$ is flat. Then there exists a canonical isomorphism
\[
\Lie(G \times_X Y / Y) \cong f^* \Lie(G/X)
\]
of $\OC_Y$-Lie algebras.
\label{it:base-change-Lie}
\end{enumerate}
\end{lem}

\begin{proof}
$(1)$
Since $G \to X$ is smooth, the $\OC_G$-module $\Omega_{G/X}$ is locally free of finite rank (see~\cite[Tag 02G1]{stacks-project}), which implies that $\omega_{G/X}$ and $\Lie(G/X)$ are locally free of finite rank over $\OC_X$. The fact that $\LLie(G/X)$ is a vector bundle follows.

$(2)$
We have
\[
f^* \Lie(G/X) = f^* \mathscr{H} \hspace{-1pt} om_{\OC_X}(\varepsilon^* \Omega_{G/X}, \OC_X) \cong \mathscr{H} \hspace{-1pt} om_{\OC_Y}(f^* \varepsilon^* \Omega_{G/X}, \OC_Y)
\]
under our assumptions (see~\cite[Theorem~7.11]{matsumura} for the second case). Now using~\cite[Proposition~II.8.10]{hartshorne} one can easily check that $f^* \varepsilon^* \Omega_{G/X}$ is the restriction to $Y$ of $\Omega_{G \times_X Y / Y}$, which finishes the proof.
\end{proof}


\subsection{Notation and assumptions}
\label{ss:notation}

In Section~\ref{sec:fields} we will work in the following setting.

Let $\FM$ be an algebraically closed field of characteristic $\ell \geq 0$.
Let $\GB$ be a connected reductive group over $\FM$, of rank $r$. Let $\BB \subset \GB$ be a Borel subgroup and $\TB \subset \BB$ be a maximal torus. We denote by $\tg \subset \bg \subset \gg$ the Lie algebras of $\TB \subset \BB \subset \GB$. We also denote by $\UB$ the unipotent radical of $\BB$, by $\ng$ its Lie algebra, and by $\DS \GB \subset \GB$ the derived subgroup. We denote by $\BB^+$ the Borel subgroup of $\GB$ which is opposite to $\BB$ with respect to $\TB$, by $\UB^+$ its unipotent radical, and by $\bg^+$ and $\ng^+$ their respective Lie algebras.

We let $\Phi$ be the root system of $(\GB,\TB)$, and $W$ be its Weyl group. For $\alpha \in \Phi$, we denote by $\gg_\alpha \subset \gg$ the corresponding root subspace.

We denote by $\Phi^+ \subset \Phi$ the system of positive roots consisting of the $\TB$-weights in $\ng^+$, and by $\Delta$ the corresponding basis of $\Phi$. We also denote by ${\check \Phi} \subset X_*(\TB)$ the coroots of $\Phi$, and by ${\check \Delta} \subset {\check \Phi}^+ \subset {\check \Phi}$ the coroots corresponding to $\Delta$ and $\Phi^+$ respectively. We denote by $\ZM \Phi \subset X^*(\TB)$, resp.~$\ZM {\check \Phi} \subset X_*(\TB)$, the lattice generated by the roots, resp.~coroots. For $\alpha \in \Phi$, resp.~${\check \alpha} \in {\check \Phi}$, we denote by $d\alpha$, resp.~$d{\check \alpha}$, the differential of $\alpha$, resp.~${\check \alpha}$, considered as an element in $\tg^*$, resp.~$\tg$.

The results of Section~\ref{sec:fields} will be proved under one of the following conditions:
\begin{itemize}
\item[(C1)]
for all $\alpha \in \Phi$, $d\alpha \neq 0$;
\item[(C2)]
$\ell$ is good for $\GB$, and $X_*(\TB)/\ZM {\check \Phi}$ has no $\ell$-torsion;
\item[(C3)]
$\ell$ is good for $\GB$, and neither $X_*(\TB)/\ZM {\check \Phi}$ nor $X^*(\TB)/\ZM \Phi$ has $\ell$-torsion;
\item[(C4)]
$\ell$ is good for $\GB$, $X_*(\TB)/\ZM {\check \Phi}$ has no $\ell$-torsion, and there exists a $\GB$-equi\-variant isomorphism $\gg \simto \gg^*$.
\end{itemize}

We claim that
\[
{\rm (C4)} \Rightarrow {\rm (C3)} \Rightarrow {\rm (C2)} \qquad \text{and} \qquad {\rm (C3)} \Rightarrow{\rm (C1)}.
\]
Indeed, the condition that $X^*(\TB)/\ZM \Phi$ has no $\ell$-torsion means that the vectors $d\alpha$ ($\alpha \in \Delta$) are linearly independent in $\tg^* \cong \FM \otimes_{\ZM} X^*(\TB)$. This implies in particular that they are non zero. Since any root is $W$-conjugate to a simple root, this justifies the implication ${\rm (C3)} \Rightarrow{\rm (C1)}$.

The implication ${\rm (C3)} \Rightarrow {\rm (C2)}$ is obvious. Now, if $\kappa \colon \gg \simto \gg^*$ is a $\GB$-equivariant isomorphism, restricting to $\TB$-fixed points we obtain an isomorphism $\tg \simto \tg^*$. It is not difficult to deduce from the $\GB$-equivariance of $\kappa$ that this isomorphism must send each differential of a simple root to a non-zero multiple of the differential of the corresponding coroot. Hence the former are linearly independent iff the latter are linearly independent, which justifies the implication ${\rm (C4)} \Rightarrow {\rm (C3)}$.

Let us note also that the condition that $X_*(\TB)/\ZM {\check \Phi}$ has no $\ell$-torsion is automatic if $\DS \GB$ is simply connected. Indeed,
%
%
let $\TB'$ be the maximal torus of $\DS \GB$ contained in $\TB$. Then, as explained in~\cite[\S II.1.18]{jantzen-gps}, we have
\[
X_*(\TB') = X_*(\TB) \cap \bigl( \sum_{{\check \alpha} \in {\check \Phi}} \QM \cdot {\check \alpha} \bigr).
\]
In particular, $X_*(\TB')/\ZM {\check \Phi}$ is the torsion submodule of $X_*(\TB)/\ZM {\check \Phi}$. 
If $\DS \GB$ is simply connected then $X_*(\TB')/\ZM {\check \Phi}=0$, hence $X_*(\TB)/\ZM {\check \Phi}$ is a free $\ZM$-module.
This remark shows that condition (C4) holds (hence also the other conditions) provided $\GB$ satisfies Jantzen's ``standard hypotheses'', see e.g.~\cite[\S 6.4]{jantzen-survey}.


\begin{remark}
\begin{enumerate}
\item
Condition (C1) is discussed in~\cite[\S 13.3]{jantzen}. The only cases when this condition might not be satisfied occur when $\ell=2$ and $\GB$ has a component of type $C_n$ ($n \geq 1$). In particular, it is satisfied if $\ell$ is good for $\GB$ except possibly when $\ell=2$ and $\GB$ has a component of type $A_1$.
\item
Primes satisfying (C3) are called ``pretty good'' in~\cite{herpel}. See also~\cite[Theorem~5.2]{herpel} for the relation between this condition and other variants of the ``standard hypotheses.'' In the case $\GB$ is semisimple, this condition is equivalent to $\ell$ being very good, and (C4) is equivalent to (C3), see the proof of Lemma~\ref{lem:bilinear-form} below.
\end{enumerate}
\end{remark}

\subsection{Regular elements and the adjoint quotient}
\label{ss:adjoint-quotient}

We continue with the notation of~\S\ref{ss:notation}. For any closed subgroup $\HB \subset \GB$ and any $x \in \gg$, we denote by $\HB_x$ the scheme-theoretic centralizer of $x$ in $\HB$, i.e.~the closed subgroup of $\HB$ defined by the fiber product $x \times_{\gg} \HB$, where the morphism $\HB \to \gg$ sends $h$ to $h \cdot x$. If $\HB$ is reduced, we also denote by $C_\HB(x)$ the reduced part of $\HB_x$, i.e.~the reduced subgroup in $\HB$ whose closed points are $\{g \in \HB \mid g \cdot x = x\}$. If $\hg$ is the Lie algebra of $\HB$, we set $\hg_x:=\{y \in \hg \mid [x,y]=0\}$.

Recall that for $x \in \GB$ we have $\dim(\GB_x) \geq r$. (This follows, by standard arguments, from the fact that any element in $\gg$ is contained in the Lie algebra of a Borel subgroup, see~\cite[Proposition~14.25]{borel}, and that $\BB$ acts trivially on $\bg/\ng$.)
An element $x \in \gg$ is called \emph{regular} if $\dim(\GB_x)=r$. We denote by $\gg_{\reg} \subset \gg$ the subset consisting of regular elements; it is open (see e.g.~\cite[Proposition~1.4]{humphreys}) and non empty (see Lemma~\ref{lem:e-regular} below). We will also consider the subset $\gg_\rs \subset \gg_\reg$ consisting of regular \emph{semi-simple} elements. This subset is open, and non-empty if (and only if) (C1) holds, see~\cite[\S 13.3]{jantzen}. For any subset $\kg \subset \gg$, we set $\kg_\reg:=\kg \cap \gg_\reg$, $\kg_\rs:=\kg \cap \gg_\rs$. Then by~\cite[\S 13.3]{jantzen}, if (C1) holds we have
\begin{equation}
\label{eqn:trs}
\tg_\rs = \{x \in \tg \mid \forall \alpha \in \Phi, \, d\alpha(x) \neq 0\}.
\end{equation}

We will denote by
\[
\chi \colon \gg \to \gg/\GB
\]
the quotient morphism, and by $\chi_{\reg} \colon \ggr \to \gg/\GB$ its restriction to regular elements. (Here, $\gg/\GB:=\Spec(\OC(\gg)^\GB)$.) In Section~\ref{sec:fields} we will need a few well-known results regarding this morphism.

\begin{lem}
\label{lem:chi-conjugate}
If $x,y \in \gg_\reg$ and $\chi(x)=\chi(y)$, then $x$ and $y$ are $\GB$-conjugate.
\end{lem}

\begin{proof}
This follows from~\cite[Proposition~7.13]{jantzen}. Alternatively, one can prove this result directly using the facts that $\chi$ separates semi-simple $\GB$-orbits (since they are closed, see~\cite[Proposition~11.8]{borel}), that $C_\GB(x)^\circ$ is reductive if $x \in \gg$ is semi-simple (see~\cite[Proposition~13.19]{borel}) and that the nilpotent cone of a connected reductive group is irreducible (see~\cite[Lemma~6.2]{jantzen}).
\end{proof}

\begin{prop}
\label{prop:Chevalley-thm}
Assume that {\rm (C1)} holds. Then
the inclusion $\tg \subset \gg$ induces an algebra isomorphism $\OC(\gg)^\GB \simto \OC(\tg)^W$; in other words an isomorphism of schemes $\tg/W \simto \gg/\GB$.
\end{prop}

\begin{proof}
In~\cite[\S 7.12]{jantzen} it is explained that the proof given in~\cite[\S II.3.17']{springer-steinberg} applies if (C1) is satisfied.
\end{proof}

When (C1) is satisfied, we will usually use Proposition~\ref{prop:Chevalley-thm} to consider $\chi$ as a morphism from $\gg$ to $\tg/W$.

\begin{lem}
\label{lem:centralizer-ss-reg}
Assume that {\rm (C1)} holds, and that $X_*(\TB)/\ZM {\check \Phi}$ has no $\ell$-torsion. Then
\begin{enumerate}
\item
the action of $W$ on $\tg_\rs$ is free (in the sense that the stabilizer of any $x \in \tg_\rs$ is trivial);
\label{it:W-trs-free}
\item
for any $x \in \tg_\rs$ we have $\GB_x=\TB$.
\label{it:stabilizer-trs}
\end{enumerate}
\end{lem}

\begin{proof}
$(1)$ Our assumptions imply that the stabilizer in $\tg$ of the reflection $s_\alpha$ associated with $\alpha \in \Phi$ is the hyperplane $\Ker(d\alpha)$. Using~\eqref{eqn:trs} we deduce that if $x \in \tg_\rs$, then $x$ is not fixed by any $s_\alpha$. Hence to conclude it is enough to prove that if $w \in W$ and $x \in \tg$ satisfy $w \cdot x = x$, then $w$ is a product of reflections stabilizing $x$. 

This property follows from a well-known argument (see e.g.~\cite[Lemma on p.~32]{humphreys}). In fact, let us first assume that $\ell >0$. Then we have natural isomorphisms
\[
\tg \cong X_*(\TB) \otimes_\ZM \FM \cong (X_*(\TB) / \ell \cdot X_*(\TB)) \otimes_{\FM_\ell} \FM.
\]
Hence it is enough to prove a similar property for the $W$-action on $X_*(\TB) / \ell \cdot X_*(\TB)$. Now let $w \in W$ and ${\check \mu} \in X_*(\TB)$ be such that $w \cdot {\check \mu} = {\check \mu} \mod \ell \cdot X_*(\TB)$. Then $w \cdot {\check \mu} - {\check \mu} \in \ZM {\check \Phi} \cap \ell \cdot X_*(\TB)$. Since $X_*(\TB)/\ZM {\check \Phi}$ has no $\ell$-torsion, this implies that $w \cdot {\check \mu} - {\check \mu} \in \ell \cdot \ZM {\check \Phi}$. Let us write $w \cdot {\check \mu} - {\check \mu} = \ell {\check \lambda}$. Then the image of ${\check \mu}$ in $X_*(\TB) \otimes_\ZM \RM$ is fixed by the element $(\ell {\check \lambda}) \cdot w$ of the group $W \ltimes \ell \ZM {\check \Phi}$, which is the affine Weyl group of~\cite[\S II.6.1]{jantzen-gps} (for the reductive group which is Langlands dual to $\GB$). By a well-known result on groups generated by reflections in real affine spaces (see~\cite[V, \S 3, Proposition~1]{bourbaki}), $(\ell {\check \lambda}) \cdot w$ is a product of reflections in $W \ltimes \ell \ZM {\check \Phi}$ stabilizing ${\check \mu}$. Projecting on $W$, we deduce that indeed $w$ is a product of reflections in $W$ which stabilize ${\check \mu} \mod \ell \cdot X_*(\TB)$, which finishes the proof in the case $\ell>0$.

The case $\ell=0$ is similar and simpler; details are left to the reader.

$(2)$
First we show that $\GB_x$ is smooth.
In fact, since $x$ is regular we have $\dim(\GB_x)
=r$. On the other hand, it is easily checked using the triangular decomposition of $\gg$ that $\gg_x=\tg$. We deduce that the inclusions $\Lie(\TB) \subset \Lie(\GB_x) \subset \gg_x$ must be equalities. We have proved that $\dim(\GB_x)=\dim(\Lie(\GB_x))$, which implies that $\GB_x$ is smooth (see~\cite[Corollary on p.~94]{waterhouse}).

It follows from~\cite[Lemma~13.3]{jantzen} that $\GB_x=C_\GB(x) \subset N_\GB(\TB)$. Since no element of $W=N_\GB(\TB)/\TB$ stabilizes $x$ by (1), we deduce that $\GB_x=\TB$.
\end{proof}

We will finally use the following easy observation.

\begin{lem}
\label{lem:kappa-centralizer}
If $\kappa \colon \gg \simto \gg^*$ is a $\GB$-equivariant isomorphism, then for any $x \in \gg$ the image $\kappa(\gg_x)$ of $\gg_x$ is the subspace $(\gg/[x,\gg])^* \subset \gg^*$.
\end{lem}

\begin{proof}
Using the $\GB$-equivariance of $\kappa$ we observe that $\kappa(\gg_x) \subset (\gg/[x,\gg])^*$. Then we conclude by a dimension argument.
\end{proof}

\section{The case of fields}
\label{sec:fields}

In this section we use the notation of \S\S\ref{ss:notation}--\ref{ss:adjoint-quotient}.

\subsection{The principal nilpotent element and the Kostant section}
\label{ss:principal-nilpotent}

For any $\alpha \in \Delta$ we choose a non-zero root vector $e_\alpha \in \gg_\alpha$, and set
\[
e:=\sum_{\alpha \in \Delta} e_\alpha.
\]

\begin{lem}
\label{lem:e-regular}
The element $e \in \gg$ is regular.
\end{lem}

\begin{proof}
First we observe that regular nilpotent elements exist: this follows from the fact that the nilpotent cone in $\gg$ has dimension $\dim(\gg)-r$ (see~\cite[Theorem~6.4]{jantzen}) and consists of finitely many $\GB$-orbits (see~\cite[\S 2.8]{jantzen} and references therein). Then the claim follows from~\cite[Lemma~5.8]{springer}. (In \emph{loc}.~\emph{cit}.~it is assumed that $\GB$ is semisimple, but this assumption does not play any role in the proof of this lemma.)
\end{proof}

\begin{lem}
\label{lem:springer}
Assume that {\rm (C2)} holds. Then the morphism
\[
\ng \to \bg, \quad x \mapsto [e,x]
\]
is injective.
\end{lem}

\begin{proof}
This result is a consequence of the considerations in~\cite[Section~2]{springer}. In fact, let $\GB_\ZM \supset \BB_\ZM \supset \TB_\ZM$ be a split connected reductive group over $\ZM$, a Borel subgroup and a (split) maximal torus such that the base change of $\GB_\ZM$, resp.~$\BB_\ZM$, resp.~$\TB_\ZM$, to $\FM$ is $\GB$, resp.~$\BB$, resp.~$\TB$. We denote by $\tg_\ZM \subset \bg_\ZM \subset \gg_\ZM$ the Lie algebras of $\TB_\ZM \subset \BB_\ZM \subset \GB_\ZM$. Then we have $\gg = \FM \otimes_\ZM \gg_\ZM$. Moreover, since the morphism
\[
\prod_{\alpha \in \Delta} \alpha \colon \TB \to (\FM^\times)^\Delta
\]
is surjective, we can assume that each $e_\alpha$ is the image in $\gg$ of a vector $x_\alpha \in \gg_\ZM$ which forms a $\ZM$-basis of the $\alpha$-weight space of $\gg_\ZM$ (with respect to the action of $\TB_\ZM$).

As in~\cite{springer} we consider the $\ZM$-grading $\gg_\ZM=\bigoplus_{i \in \ZM} \gg_\ZM^i$ induced by the height\footnote{Recall that the height of a positive root $\alpha$ is the number of simple roots occurring in the decomposition of $\alpha$ as a sum of simple roots (counted with multiplicities). The height of a nega\-tive root $\alpha$ is the opposite of the height of $-\alpha$.} of roots, and denote by $t_i \colon \gg_\ZM^i \to \gg_\ZM^{i+1}$ the morphism $y \mapsto [\sum_{\alpha \in \Delta} x_\alpha,y]$. Then the morphism of the lemma is the morphism obtained from
\[
\bigoplus_{i < 0} t_i \colon \bigoplus_{i<0} \gg_\ZM^i \to \bigoplus_{i \leq 0} \gg_\ZM^i
\]
by applying the functor $\FM \otimes_\ZM (-)$.

It is clear that $t_{-1}$ identifies with the inclusion $\ZM {\check \Phi} \hookrightarrow X_*(\TB)$. Hence the induced morphism $\FM \otimes_{\ZM} \gg^{-1}_\ZM \to \FM \otimes_{\ZM} \gg^0_\ZM$ is injective if (C2) holds (more precisely, iff the second condition in (C2) holds).

Now if $i<-1$, then by~\cite[Proposition~2.2]{springer} the morphism $t_i$ is injective, and by~\cite[Theorem~2.6]{springer} its cokernel has no $p$-torsion if $p$ is good. (In~\cite[Section 2]{springer} it is assumed that the root system of $\GB_\ZM$ is simple. However the result we need follows from Springer's results applied to each simple factor in a simply-connected cover of the derived subgroup of $\GB_\ZM$.) Hence $\FM \otimes_\ZM t_i$ is injective, which finishes the proof.
\end{proof}

From now on we will assume that (C2) holds, so that we can apply Lemma~\ref{lem:springer}.
Let us consider the cocharacter ${\check \lambda}_\circ:= \sum_{{\check \alpha} \in {\check \Phi}^+} {\check \alpha} \colon \Gm \to \TB$. This cocharacter defines (via the adjoint action) a $\Gm$-action on $\gg$. 
The vector $e$ is a weight vector for this action, of weight $2$.
Let us choose a $\Gm$-stable complement $\sg$ to $[e,\ng]$ in $\bg$, and set
\[
\SC:=e + \sg, \qquad \Upsilon:=e + \bg.
\]
Then $\Upsilon$ is endowed with a $\Gm$-action defined by $t \cdot x:=t^{-2} {\check \lambda}_\circ(t) \cdot x$. This action contracts $\Upsilon$ to $e$ (as $t \to \infty$) and stabilizes $\SC$.

Let us note that 
\begin{equation}
\label{eqn:Up-reg}
\Upsilon \subset \ggr.
\end{equation}
Indeed,
by Lemma~\ref{lem:e-regular}, $e \in \ggr$. Hence an open neighborhood of $e$ in $\Upsilon$ is included in $\ggr$. Using the contracting $\Gm$-action on $\Upsilon$ defined above, we deduce that the whole of $\Upsilon$ is included in $\ggr$.

\begin{lem}
\label{lem:direct-sum}
Assume that {\rm (C3)} holds. Then we have
\[
\gg = \sg \oplus [e, \gg].
\]
\end{lem}

\begin{proof}
By the same argument as for centralizers in $\GB$ (see~\S\ref{ss:adjoint-quotient}), one can check that $\dim(\gg_x) \geq r$ for all $x \in \gg$. In particular, it follows that $\dim([e,\gg]) \leq \dim(\GB)-r$. Hence to prove the lemma it suffices to prove that $\gg=\sg + [e,\gg]$. And for this it suffices to prove that $\ng^+ \subset [e,\gg]$.

This fact will again by deduced from~\cite{springer}. We use the same notation as in the proof of Lemma~\ref{lem:springer}. As in this proof, one can assume that each $e_\alpha$ is obtained from a $\ZM$-basis of the corresponding root space in $\gg_\ZM$, and consider the morphisms $t_i$. Then $\FM \otimes t_0 \colon \FM \otimes_\ZM \gg^0_\ZM \to \FM \otimes_\ZM \gg^1_\ZM$ identifies with the morphism
\[
\tg \to \FM^\Delta, \quad t \mapsto (d\alpha(t))_{\alpha \in \Delta}.
\]
Hence it is surjective if and only if the linear forms $d\alpha$ ($\alpha \in \Delta$) are linearly independent, i.e.~if and only if the third condition in {\rm (C3)} holds.

And for $i \geq 1$, combining~\cite[Proposition~2.2]{springer} and~\cite[Theorem~2.6]{springer} we obtain that $\Coker(t_i)$ is finite and has no $p$-torsion if $p$ is good. Hence $\FM \otimes_\ZM t_i$ is surjective, which finishes the proof.
\end{proof}

\begin{remark}
\label{rk:direct-sum-S}
More generally, if (C3) holds, for any $x \in \Upsilon$ we have $\gg = \sg \oplus [x,\gg]$. (In particular, we have $\dim(\gg_x)=r$ in this case.) Indeed, as in the proof of Lemma~\ref{lem:direct-sum}, it suffices to prove that $\gg = \sg + [x,\gg]$, i.e.~that the morphism $\sg \oplus \gg \to \gg$ defined by $(s,y) \mapsto s + [x,y]$ is surjective. This property is an open condition on $x$. Since it is satisfies by $e$, it is satisfied in a neighborhood of $e$ in $\Upsilon$. We conclude using the contracting $\Gm$-action.
\end{remark}

\subsection{Kostant's theorem}
\label{ss:Kostant-section-F}

\begin{prop}
\label{prop:ganginzburg}
Assume that {\rm (C2)} holds. Then
the morphism
\[
\UB \times \SC \to \Upsilon
\]
induced by the adjoint action is an isomorphism of varieties.
\end{prop}

\begin{proof}
The proof is copied from that of~\cite[Lemma~2.1]{ganginzburg}. 

Let us denote by $\psi$ the morphism of the lemma.
First, we remark that the differential $d_{(1,e)}\psi$ of $\psi$ at $(1,e)$ can be identified with the morphism
\[
\left\{
\begin{array}{ccc}
\ng \oplus \sg & \to & \bg \\
(n,s) & \mapsto & [n,e]+s
\end{array}
\right. .
\]
By Lemma~\ref{lem:springer} and the definition of $\sg$, this morphism is an isomorphism. It follows that $\psi$ is {\'e}tale in a neighborhood of $(1,e)$, and in particular dominant. We deduce that the morphism $\psi^* \colon \OC(\Upsilon) \to \OC(\UB \times \SC)$ is injective.

To prove that $\psi^*$ is an isomorphism, we consider the $\Gm$-actions on $\SC$ and $\Upsilon$ defined in \S\ref{ss:principal-nilpotent}. In fact we also define a $\Gm$-action on $\UB \times \SC$ by setting $t \cdot (u,s) = ({\check \lambda}_\circ(t) u {\check \lambda}_\circ(t)^{-1}, t \cdot s)$.
Then the actions on $\UB \times \SC$ and $\Upsilon$ are contracting as $t \to \infty$ (to $(1,e)$ and to $e$ respectively), and $\psi$ is $\Gm$-equivariant. Hence to conclude we only have to prove that the characters of the action of $\Gm$ on $\OC(\Upsilon)$ and $\OC(\UB \times \SC)$ coincide. However, if $X$ is any of these spaces, $x$ is its $\Gm$-fixed point, $\mg_x \subset \OC(X)$ the ideal of $x$, and $T_x$ the tangent space of $X$ at $x$, then, as a $\Gm$-module, $\OC(X)$ is isomorphic to the associated graded of the $\mg_x$-adic filtration, which is itself isomorphic (again as a $\Gm$-module) to $\OC(T_x)$. Since $d_{(1,e)}\psi$ is an isomorphism, we deduce the equality of characters, which finishes the proof.
\end{proof}


Now we can prove the main result of this subsection. Our proof is essentially identical to that in~\cite[Proposition~6.3]{veldkamp} (which is due to Springer), and is partially based on the same idea as for Proposition~\ref{prop:ganginzburg}. The fact that the proof in~\cite{veldkamp} can be generalized to very good primes using the results of Demazure is mentioned in~\cite[\S 3.14]{slodowy}, see also~\cite[\S 7.14]{jantzen}. 
In fact, slightly weaker assumptions are sufficient.

\begin{thm}[Kostant's theorem]
\label{thm:kostant}
Assume that {\rm (C1)} and {\rm (C2)} hold.
Then
the natural morphisms
\[
\SC \to \Upsilon/\UB \to \gg / \GB
\]
are isomorphisms.
\end{thm}

\begin{proof}
Proposition~\ref{prop:ganginzburg} implies that the first morphism is an isomorphism. So to complete the proof it is enough to prove that the composition $\SC \to \gg/\GB$ is an isomorphism. First, let us prove that this morphism is dominant. 

For this it suffices to prove that $\GB \cdot \SC$ contains the open set of regular semisimple elements. So, let $x \in \gg$ be a regular semisimple element. Then $x$ is $\GB$-conjugate to a regular element $y \in \tg$ by~\cite[Proposition~11.8]{borel}. Now $\UB^+ \cdot y = y+\ng^+ \ni y + e$ (see e.g.~\cite[\S 13.3]{jantzen}), hence $x$ is $\GB$-conjugate to $y+e$. Finally, since $y+e \in e + \bg=\Upsilon$, it follows from Proposition~\ref{prop:ganginzburg} that $x \in \GB \cdot \SC$.

Now we consider again the $\Gm$-action on $\SC$ defined in~\S\ref{ss:principal-nilpotent}. Then the morphism $\SC \to \gg/\GB$ is $\Gm$-equivariant, where the $\Gm$-action on $\gg/\GB$ is induced by the action $t \cdot x=t^{-2} x$ on $\gg$. As in the proof of Proposition~\ref{prop:ganginzburg}, to conclude it is enough to prove that the $\Gm$-modules $\OC(\SC)$ and $\OC(\gg/\GB)$ have the same character. 

To prove this we remark that $\SC$ is isomorphic, as a $\Gm$-variety, to $\sg$, hence to $\bg/[e,\ng]$. (Here the $\Gm$-action on $\bg$ is given by $t \cdot x=t^{-2} {\check \lambda}_\circ(t) \cdot x$.) Now recall the notation introduced in the proof of Lemma~\ref{lem:springer}. As in this lemma we can assume that each $e_\alpha$ is obtained from an element $x_\alpha \in \gg_\ZM$. Then one can consider the base change $\GB_\CM$ of $\GB_\ZM$ to $\CM$, and the corresponding objects $\BB_\CM$, $\TB_\CM$, $\gg_\CM$, $\bg_\CM$, $\tg_\CM$. If $e_\CM$ denotes the image of $\sum_{\alpha \in \Delta} x_\alpha$ in $\gg_\CM$, then it follows from the proof of Lemma~\ref{lem:springer} that the $\Gm$-character of $\bg/[e,\ng]$ coincides with the $(\Gm)_\CM$-character of $\bg_\CM/[e_\CM, \ng_\CM]$ (where the $(\Gm)_\CM$-action on $\bg_\CM$ is defined by the same formula as the $(\Gm)_\FM$-action on $\bg_\FM$ considered above). By Lemma~\ref{lem:direct-sum} applied to the field $\CM$, we have a $\Gm$-equivariant isomorphism $\bg_\CM/[e_\CM, \ng_\CM] \simto \gg_\CM/[e_\CM,\gg_\CM]$. Using Kostant's theorem over $\CM$ (see~\cite[Theorem 7]{kostant}), we deduce that the $\Gm$-character of $\OC(\SC)$ coincides with the $(\Gm)_\CM$-character of $\OC(\tg_\CM/W)$, where the $(\Gm)_\CM$-action is induced by the action on $\tg_\CM$ given by $t \cdot x=t^{-2} x$.

From these considerations we obtain that to conclude the proof it suffices to prove that the $\Gm$-character of $\OC(\gg/\GB)$ coincides with the $(\Gm)_\CM$-character of $\OC(\tg_\CM/W)$. By Proposition~\ref{prop:Chevalley-thm}, it suffices to show that the $\Gm$-character of $\OC(\tg/W)$ coincides with the $(\Gm)_\CM$-character of $\OC(\tg_\CM/W)$. However, if $N$ is the product of the prime numbers $p$ which are bad primes for $\GB$ or such that $X_*(\TB)/\ZM{\check \Phi}$ has $p$-torsion, then we have
\begin{align*}
\OC(\tg_\CM/W) & = \CM \otimes_{\ZM[1/N]} \mathrm{S}_{\ZM[1/N]}(X^*(\TB) \otimes_\ZM \ZM[1/N])^W \\
\OC(\tg/W) & = \FM \otimes_{\ZM[1/N]} \mathrm{S}_{\ZM[1/N]}(X^*(\TB) \otimes_\ZM \ZM[1/N])^W
\end{align*}
by~\cite[Corollaire on p.~296]{demazure}. (Note that the $\ZM$-module $\Coker(i)$, where $i$ is the map considered in~\cite[\S 5]{demazure} for the root system $\Phi \subset X^*(\TB)$, is isomorphic to $\Ext^1_{\ZM}(X_*(\TB)/\ZM{\check \Phi} , \ZM)$ with our notation; in particular, this $\ZM$-module has the same torsion as $X_*(\TB)/\ZM{\check \Phi}$.)
By~\cite[Th{\'e}or{\`e}me~2 on p.~295]{demazure}, the $\ZM[1/N]$-module $\mathrm{S}_{\ZM[1/N]}(X^*(\TB) \otimes_\ZM \ZM[1/N])^W$ is graded free. We deduce the equality of characters, which finishes the proof.
\end{proof}

\begin{remark}
If (C1) and (C2) hold,
it follows from Theorem~\ref{thm:kostant} that $\gg/\GB = \tg/W$ is smooth. In fact this follows more directly from the ingredients of the proof: see~\cite[Corollaire and Th{\'e}or{\`e}me 3 on p.~296]{demazure}.
\end{remark}

\subsection{The universal centralizer}
\label{ss:centralizer-F}

The main object of study in this paper is the ``universal centralizer'' over $\gg$, i.e.~the closed subgroup of the group scheme $\GB \times \gg$ over $\gg$ defined as the fiber product
\[
\IB := \gg \times_{\gg \times \gg} (\GB \times \gg).
\]
Here the morphism $\gg \to \gg \times \gg$ is the diagonal embedding, and the morphism $\GB \times \gg \to \gg \times \gg$ is defined by $(g,x) \mapsto (g \cdot x, x)$. By definition, for $x \in \gg$, the (scheme-theoretic) fiber of $\IB$ over $x$ identifies with the group scheme $\GB_x$.

We will denote by $\IB_{\reg}$ the restriction of $\IB$ to $\gg_{\reg}$, by $\IB_{\rs}$ its restriction to $\gg_{\rs}$, and by $\IB_{\SC}$ its restriction to $\SC$.

\begin{lem}
\label{lem:a-smooth}
Assume that {\rm (C3)} holds.
Then
the morphism
\[
a \colon \GB \times \SC \to \ggr
\]
induced by the adjoint action is smooth and surjective.
\end{lem}

\begin{proof}
Let us first prove smoothness.
The differential of $a$ at $(1,e)$ can be identified with the morphism $\gg \times \sg \to \gg$ sending $(x,s)$ to $[x,e]+s$. By Lemma~\ref{lem:direct-sum} this differential is surjective under our assumptions, hence $a$ is smooth in a neighborhood of $(1,e)$. In fact, since this morphism is $\GB$-equivariant, there exists a neighborhood $V$ of $e$ in $\SC$ such that $a$ is smooth on $\GB \times V$. Then using the $\Gm$-action on $\SC$ as in the proof of~\eqref{eqn:Up-reg}, we conclude that $a$ is smooth.

To prove surjectivity, using smoothness it is enough to prove that closed points of $\ggr$ belong to the image of $a$. In other words, we have to prove that any element of $\gg_\reg$ is $\GB$-conjugate to an element of $\SC$.  Let $x \in \gg_\reg$. Then it follows from Theorem~\ref{thm:kostant} that there exists $y \in \SC$ such that $\chi(x)=\chi(y)$. We conclude using Lemma~\ref{lem:chi-conjugate} and~\eqref{eqn:Up-reg}.
\end{proof}

\begin{remark}
\label{rk:dim-centralizer}
Assume that (C3) holds. Then for all $x \in \gg_\reg$ we have $\dim(\gg_x)=r$. In fact, this property was observed for $x \in \SC$ in Remark~\ref{rk:direct-sum-S}. But since any element of $\gg_\reg$ is $\GB$-conjugate to an element in $\SC$ by Lemma~\ref{lem:a-smooth}, it holds on the whole of $\gg_\reg$. In the terminology of~\cite[\S 5.7]{springer} this means that, in this case, all regular elements in $\gg$ are \emph{smoothly regular}.
\end{remark}

Let us note the following consequence. (See also~\cite[Corollary on p.~746]{bg} for a different proof, assuming that Jantzen's ``standard hypotheses'' hold and that $\ell$ is odd.)

\begin{prop}
\label{prop:chi-reg-smooth}
Assume that {\rm (C3)} holds. Then
the morphism $\chi_{\reg} \colon \gg_\reg \to \gg/\GB$ is smooth and surjective.
\end{prop}

\begin{proof}
Consider the composition
\[
\GB \times \SC \xrightarrow{a} \ggr \xrightarrow{\chi_\reg} \gg/\GB.
\]
Using Theorem~\ref{thm:kostant} this morphism identifies with the projection $\GB \times \SC \to \SC$ on the second factor. In particular it is smooth and surjective. By~\cite[Tag 02K5]{stacks-project}, we deduce that $\chi_\reg$ is smooth and surjective. 
\end{proof}

\begin{remark}
\label{rk:tangent-space}
Assume that (C3) holds.
The fact that $\chi_\reg$ is smooth implies that, for all $x \in \gg_\reg$, the morphism $d_x(\chi_\reg) \colon \TM_x(\gg_\reg) \to \TM_{\chi(x)}(\gg/\GB)$ is surjective; in other words we have a surjection $\gg \twoheadrightarrow \TM_{\chi(x)}(\gg/\GB)$. (Here $ \TM_x(\gg_\reg)$, resp.~$\TM_{\chi(x)}(\gg/\GB)$, denotes the tangent space to $\gg_\reg$ at $x$, resp.~to $\gg/\GB$ at $\chi(x)$.) Since the composition of $\chi_\reg$ with the morphism $\GB \to \gg_\reg$ defined by $g \mapsto g \cdot x$ is constant, the differential $d_x(\chi_\reg)$ must vanish on $[x,\gg]$. Since $\dim(\gg/[x,\gg])=r$ (see Remark~\ref{rk:dim-centralizer}), we finally obtain that $d_x(\chi_\reg)$ induces an isomorphism $\gg/[x,\gg] \simto \TM_{\chi(x)}(\gg/\GB)$.
\end{remark}

\begin{prop}
\label{prop:IS-smooth}
Assume that {\rm (C3)} holds. Then
the group scheme $\IB_\SC$ is smooth over $\SC$.
\end{prop}

\begin{proof}
By definition we have
$\IB_\SC = \SC \times_{\gg \times \SC} (\GB \times \SC)$. Since each of the morphisms which define this fiber product factors through $\gg \times_{\gg/\GB} \SC$, we also have
$\IB_\SC \cong \SC \times_{\gg \times_{\gg/\GB} \SC} (\GB \times \SC)$.
Now using Theorem~\ref{thm:kostant} we obtain that the first projection induces an isomorphism $\gg \times_{\gg/\GB} \SC \simto \gg$, which provides an isomorphism
\begin{equation}
\label{eqn:isom-IS}
\IB_\SC \simto \SC \times_{\gg} (\GB \times \SC).
\end{equation}
In this fiber product the morphism $\GB \times \SC \to \gg$ is the composition of the morphism $a$ of Lemma~\ref{lem:a-smooth} with the open embedding $\ggr \hookrightarrow \gg$, hence it is smooth; this implies our claim.
%
%
\end{proof}

\begin{cor}
\label{cor:Ireg-smooth}
Assume that {\rm (C3)} holds. Then
the group scheme $\IB_\reg$ is smooth over $\gg_{\reg}$.
\end{cor}

\begin{proof}
The following diagram (where each map is the natural one) is Cartesian by the $\GB$-equivariance of $\IB$:
\[
\xymatrix@C=1.2cm@R=0.5cm{
G \times \IB_\SC \ar[r] \ar[d] & \IB_\reg \ar[d] \\
G \times \SC \ar[r]^-{a} & \gg_\reg.
}
\]
By Lemma~\ref{lem:a-smooth} and Proposition~\ref{prop:IS-smooth}, the composition $G \times \IB_\SC \to \gg_\reg$ is smooth. Since $a$ is smooth and surjective, so is the morphism $G \times \IB_\SC \to \IB_\reg$. Using~\cite[Tag 02K5]{stacks-project}, we deduce that the morphism $\IB_\reg \to \gg_\reg$ is smooth.
\end{proof}

\begin{remark}
\label{rk:centralizer-elements-S}
Assume that (C3) holds. Then
it follows from Corollary~\ref{cor:Ireg-smooth} that for any $x \in \gg_\reg$, the stabilizer $\GB_x$ is a smooth group scheme over $\FM$, i.e.~an algebraic group in the ``traditional'' sense. In particular we have $\dim_\FM (\Lie(\GB_x)) = \dim(\GB_x)=r$, see~\cite[Second corollary on p.~94]{waterhouse}, and therefore the embedding $\Lie(\GB_x) \hookrightarrow \gg_x$ is an isomorphism (see Remark~\ref{rk:dim-centralizer}).
\end{remark}

\begin{cor}
\label{cor:Ireg-commutative}
Assume that {\rm (C3)} holds. Then
the group scheme $\IB_\reg$ is commutative.
\end{cor}

\begin{proof}
By Corollary~\ref{cor:Ireg-smooth}, $\IB_\reg$ is flat over $\gg_{\reg}$. Hence it is enough to prove that $\IB_\rs$ is commutative. Since any element in $\gg_\rs$ is $\GB$-conjugate to an element in $\tg_\rs$ (see~\cite[Proposition~11.8]{borel}), it is enough to prove that for any $x \in \tg_\rs$, the group $\GB_x$ is commutative. However in this case we have $\GB_x=\TB$ (see Lemma~\ref{lem:centralizer-ss-reg}\eqref{it:stabilizer-trs}), which is clearly commutative.
\end{proof}

Let us also note the following property.

\begin{prop}
\label{prop:J}
Assume that {\rm (C3)} holds. Then
there exists a unique smooth affine commutative group scheme $\JB$ over $\tg/W$ such that the pullback of $\JB$ under $\chi_\reg \colon \gg_\reg \to \tg/W$ is $\IB_\reg$.
\end{prop}

\begin{proof}[Sketch of proof]
The group scheme $\JB$ is constructed by descent (for schemes affine over a base) along the smooth and surjective morphism $\chi_\reg$ (see Proposition~\ref{prop:chi-reg-smooth}), following the arguments in~\cite[Lemma~2.1.1]{ngo}.

These arguments use the fact that the morphism $\mu \colon G \times \gg_\reg \to \gg_\reg \times_{\tg/W} \gg_\reg$ defined by $(g,x) \mapsto (x,g \cdot x)$ is smooth and surjective. To check this property, we consider the composition
\begin{equation}
\label{eqn:morphism-ngo}
\GB \times \GB \times \SC \xrightarrow{\mathrm{id}_\GB \times a} \GB \times \gg_\reg \xrightarrow{\mu} \gg_\reg \times_{\tg/W} \gg_\reg.
\end{equation}
To prove that $\mu$ is smooth and surjective, it suffices to prove that this composition has the same properties: then surjectivity is clear, and smoothness follows from Lemma~\ref{lem:a-smooth} and~\cite[Tag 02K5]{stacks-project}. Now, to prove that~\eqref{eqn:morphism-ngo} is smooth and surjective, it is enough to prove that its composition with the automorphism of $\GB \times \GB \times \SC$ defined by $(g,h,x) \mapsto (hg^{-1},g,x)$ has the same properties. The latter morphism sends $(g,h,x)$ to $(g \cdot x, h \cdot x)$. Hence it identifies with the smooth and surjective morphism
\[
(\GB \times \SC) \times_{\tg/W} (\GB \times \SC) \xrightarrow{a \times a} \gg_\reg \times_{\tg/W} \gg_\reg,
\]
which finishes the proof.
\end{proof}

\begin{remark}
\label{rk:J-S}
The group scheme $\IB_\SC$ is naturally isomorphic to the pullback of $\JB$ along the isomorphism $\SC \simto \tg/W$ given by
the composition $\SC \hookrightarrow \gg_\reg \xrightarrow{\chi_\reg} \tg/W$.
\end{remark}

We finish this subsection with the construction of the ``Kostant--Whittaker reduction'' functor.\footnote{In~\cite{mr} we use this terminology for a slightly different functor.} This functor turns out to be very useful to construct equivalences of categories; see~\cite{bf, dodd, mr} and Section~\ref{sec:derived-Satake}.
We denote by $\Rep(\JB)$, resp.~$\Rep(\IB_{\reg})$, the category of representations of $\JB$, resp.~$\IB_\reg$, which are coherent over $\OC_{\tg/W}$, resp.~$\OC_{\gg_\reg}$.

\begin{prop}
\label{prop:coh-greg}
Assume that {\rm (C3)} holds.
There exists a natural equivalence of abelian categories
\[
\Coh^{\GB}(\gg_\reg) \simto \Rep(\JB).
\]
\end{prop}

\begin{proof}
We consider the following composition of functors:
\begin{equation*}
\Coh^{\GB}(\gg_\reg) \to \Rep(\IB_\reg) \to \Rep(\IB_\SC) \simto \Rep(\JB).
\end{equation*}
Here the first functor is the natural ``forgetful'' functor, see e.g.~\cite[\S 2.2]{mr}. The second arrow is induced by restriction to $\SC$. And the third functor is induced by the isomorphism $\SC \simto \tg/W$, see Remark~\ref{rk:J-S}. We will prove that the composition of the first two functors is an equivalence, which will imply the proposition.

Recall that $a \colon \GB \times \SC \to \gg_\reg$ is smooth and surjective (see Lemma~\ref{lem:a-smooth}), in particular faithfully flat and quasi-compact. Hence descent theory implies that the category $\Coh^{\GB}(\gg_\reg)$ is equivalent to the category of pairs $(\FC, \vartheta)$ where $\FC$ is an object of $\Coh^{\GB}(\GB \times \SC)$ and $\vartheta$ is an isomorphism of $\GB$-equivariant coherent sheaves between the two natural pullbacks of $\FC$ to $(\GB \times \SC) \times_{\gg_\reg} (\GB \times \SC)$, satisfying the descent conditions. (Here, $\GB$ acts diagonally on the fiber product.) Now we have canonical isomorphisms
\[
(\GB \times \SC) \times_{\gg_\reg} (\GB \times \SC) \cong \GB \times \bigl( \SC \times_{\gg_\reg} (\GB \times \SC) \bigr) \cong \GB \times \IB_\SC.
\]
(Here the first isomorphism is induced by the $\GB$-action and follows from $\GB$-equiva\-riance, and the second isomorphism is given by~\eqref{eqn:isom-IS}.) Under these isomorphisms, the first, resp.~second, projection $(\GB \times \SC) \times_{\gg_\reg} (\GB \times \SC) \to \GB \times \SC$ identifies with the morphism $(g, (h,s)) \mapsto (g,s)$, resp.~$(g,(h,s)) \mapsto (gh, s)$. Hence $\vartheta$ can be interpreted as an isomorphism between the pullbacks of $\FC$ to $\GB \times \IB_\SC$ under these two morphisms.

Now we observe that there exist canonical equivalences of categories
\[
\Coh(\SC) \simto \Coh^{\GB}(\GB \times \SC), \qquad \Coh(\IB_\SC) \simto \Coh^{\GB}(\GB \times \IB_\SC)
\]
induced by pullback under the second projection. Under these equivalences, both pullbacks considered above identify with the natural pullback morphism $\Coh(\SC) \to \Coh(\IB_\SC)$. Hence $\Coh^{\GB}(\gg_\reg)$ is equivalent to the category of pairs $(\FC', \vartheta')$ where $\FC'$ is an object of $\Coh(\SC)$ and $\vartheta'$ is an automorphism of the pullback of $\FC'$ to $\IB_\SC$ satisfying certain conditions. But $\Coh(\SC)$ is equivalent to the category of $\OC(\SC)$-modules via $\Gamma(\SC,-)$, and unravelling the definitions one can check that the datum of the automorphism $\vartheta'$ considered above is equivalent to the datum of an $\OC(\IB_\SC)$-comodule structure on $\Gamma(\SC,\FC')$, which finishes the proof.
%
%
\end{proof}

\begin{remark}
\label{rk:KW-reduction-canonical}
One can describe the equivalence of Proposition~\ref{prop:coh-greg} in slightly more canonical terms, as the composition $\Coh^{\GB}(\gg_\reg) \to \Rep(\IB_\reg) \to \Rep^\UB(\IB_\Upsilon) \simto \Rep(\JB)$. Here $\Rep^\UB(\IB_\Upsilon)$ is the category of $\UB$-equivariant representations of the restriction $\IB_\Upsilon$ of $\IB$ to $\Upsilon$, and the last functor is induced by the isomorphism $\Upsilon / \UB \simto \tg/W$, see~Theorem~\ref{thm:kostant}. In particular, this functor does not depend on the choice of $\sg$, up to canonical isomorphism.
\end{remark}

\subsection{Lie algebras}
\label{ss:Lie-centralizer-F}

We set 
\[
\IG:=\Lie(\IB/\gg), \ \ \IG_\reg:=\Lie(\IB_\reg/\gg_\reg), \ \
\IG_{\SC}:=\Lie(\IB_\SC/\SC), \ \ \JG:=\Lie(\JB/(\tg/W)).
\]
It follows from Lemma~\ref{lem:Lie-alg}\eqref{it:base-change-Lie} that $\IG_\reg$, resp.~$\IG_{\rs}$, is the restriction of $\IG$ to $\gg_\reg$, resp.~$\gg_\rs$. We also set
\[
\IM:=\LLie(\IB/\gg), \ \ \IM_\reg:=\LLie(\IB_\reg/\gg_\reg), \ \ 
\IM_{\SC}:=\LLie(\IB_\SC/\SC), \ \ \JM:=\LLie(\JB/(\tg/W)).
\]

%
%
%

The following properties are easy consequences of the results of \S\ref{ss:centralizer-F}.

\begin{prop}
\label{prop:properties-Lie-alg}
Assume that {\rm (C3)} holds. Then:

\begin{enumerate}
\item
the coherent sheaf $\IG_\reg$, 
resp.~$\IG_\SC$, resp.~$\JG$, on $\gg_\reg$, 
resp.~$\SC$ resp.~$\tg/W$, is locally free of finite rank;
\item
the Lie algebra $\IM_\reg$, 
resp.~$\IM_\SC$, resp.~$\JM$, is a vector bundle over $\gg_\reg$, 
resp.~$\SC$, resp.~$\tg/W$;
\item
the Lie algebras $\IG_\reg$, 
$\IG_\SC$ and $\JG$ are commutative;
\item
the restriction of $\IG_\reg$ to $\SC$ is canonically isomorphic to $\IG_\SC$; in other words there exists a canonical Cartesian diagram
\[
\xymatrix@R=0.5cm{
\IM_\SC \ar[r] \ar[d] & \IM_\reg \ar[d] \\
\SC \ar[r] & \gg_\reg;
}
\]
\item
there exists a canonical isomorphism $\chi_\reg^*(\JG) \cong \IG_\reg$; in other words there exists a canonical Cartesian diagram
\[
\xymatrix@R=0.5cm{
\IM_\reg \ar[r] \ar[d] & \JM \ar[d] \\
\gg_\reg \ar[r]^-{\chi_\reg} & \tg/W;
}
\]
\item
the coherent sheaf $\IG_\SC$ is canonically isomorphic to the pullback of $\JG$ under the isomorphism $\SC \simto \tg/W$; in other words $\IM_\SC$ is canonically isomorphic to the pullback of $\JM$ under the isomorphism $\SC \simto \tg/W$.
\end{enumerate}
\end{prop}

\begin{proof}
$(1)$ and $(2)$ follow from Lemma~\ref{lem:Lie-alg}\eqref{it:Lie-smooth} together with Proposition~\ref{prop:IS-smooth}, Corollary~\ref{cor:Ireg-smooth} and Proposition~\ref{prop:J}. $(3)$ follows from Corollary~\ref{cor:Ireg-commutative} and Proposition~\ref{prop:J}. $(4)$ follows from Lemma~\ref{lem:Lie-alg}\eqref{it:base-change-Lie} and Corollary~\ref{cor:Ireg-smooth}. $(5)$ follows from Lemma~\ref{lem:Lie-alg}\eqref{it:base-change-Lie} and Proposition~\ref{prop:J}. Finally, $(6)$ follows from Remark~\ref{rk:J-S}.
\end{proof}

The main result of this subsection gives a more concrete description of the Lie algebras $\JM$, $\IM_\reg$ and $\IM_\SC$. The proof will require that (C4) holds. In fact, from now on we fix a $\GB$-equivariant isomorphism $\kappa \colon \gg \simto \gg^*$.

\begin{thm}
\label{thm:Lie-alg-cotangent}
Assume that {\rm (C4)} holds. Then
there exist canonical isomorphisms
\[
\JM \cong \TM^*(\tg/W), \qquad \IM_\reg \cong \gg_\reg \times_{\tg/W} \TM^*(\tg/W), \qquad \IM_\SC \cong \TM^*(\SC)
\]
of commutative Lie algebras over $\tg/W$, $\gg_\reg$ and $\SC$, respectively.
\end{thm}

\begin{proof}
It is enough to construct the third isomorphism; then the other two isomorphisms follow, using Proposition~\ref{prop:properties-Lie-alg}(5)-(6).

First we construct a morphism of coherent sheaves $\IG_\SC \to \Omega_{\SC}$. Since $\IB_\SC$ is a closed subgroup of $\GB \times \SC$, the Lie algebra $\IG_\SC$ embeds naturally in $\Lie(\GB \times \SC/\SC)=\gg \otimes_\FM \OC_{\SC}$. Now since $\SC$ is a smooth closed subvariety in $\gg$, the tangent sheaf $\TC_{\SC}$ embeds in the restriction $\TC_{\gg|\SC}$ of $\TC_{\gg}$ to $\SC$, i.e.~in $\gg \otimes_\FM \OC_{\SC}$, and the cokernel of this embedding is a locally free sheaf. We deduce a canonical surjection $\gg^* \otimes_\FM \OC_{\SC} \twoheadrightarrow \Omega_{\SC}$. Then we define our morphism as the composition
\begin{equation}
\label{eqn:morphism-thm-cotangent}
\IG_\SC \hookrightarrow \gg \otimes_\FM \OC_{\SC} \xrightarrow[\sim]{\kappa \otimes \OC_\SC} \gg^* \otimes_\FM \OC_{\SC} \twoheadrightarrow \Omega_{\SC}.
\end{equation}
To finish the proof it suffices to prove that this morphism is an isomorphism.

Recall the contracting $\Gm$-action on $\SC$ considered in~\S\ref{ss:principal-nilpotent}. One can ``extend'' this action to a compatible action on $\GB \times \SC$ by group automorphisms, setting $t \cdot (g,x) = ({\check \lambda}_\circ(t) g {\check \lambda}_\circ(t)^{-1}, t \cdot x)$. This action stabilizes $\IB_\SC$, and we deduce a $\Gm$-equivariant structure on $\IG_\SC$.
The $\Gm$-action on $\SC$ also induces a $\Gm$-equivariant structure $\Omega_\SC$, and~\eqref{eqn:morphism-thm-cotangent} induces a morphism of $\Gm$-equivariant coherent sheaves $\IG_\SC \to \Omega_\SC \langle -2 \rangle$. Since $\Omega_\SC$ is locally free (in particular, has no torsion), using 
the graded Nakayama lemma we deduce that to prove that~\eqref{eqn:morphism-thm-cotangent} is an isomorphism it suffices to prove that the induced morphism
\begin{equation}
\label{eqn:morphism-thm-cotangent-e}
i_e^*(\IG_\SC) \to i_e^*(\Omega_\SC)
\end{equation}
is an isomorphism, where $i_e \colon \{e\} \hookrightarrow \SC$ is the inclusion.

We claim that there exists a canonical isomorphism $i_e^*(\IG_\SC) \simto \gg_e$. Indeed, using Lemma~\ref{lem:Lie-alg}\eqref{it:base-change-Lie} and Proposition~\ref{prop:IS-smooth} there exists a canonical isomorphism $i_e^*(\IG_\SC) \simto \Lie(\GB_e)$, and by Remark~\ref{rk:centralizer-elements-S} the right-hand side identifies with $\gg_e$ canonically.
Using this claim and the obvious isomorphism $i_e^*(\Omega_\SC) \cong \sg^*$, one can identify~\eqref{eqn:morphism-thm-cotangent-e} with the composition
\[
\gg_e \hookrightarrow \gg \xrightarrow[\sim]{\kappa} \gg^* \twoheadrightarrow \sg^*.
\]
The fact that this morphism is an isomorphism follows from Lemma~\ref{lem:kappa-centralizer} and Lemma~\ref{lem:direct-sum}.
\end{proof}

\subsection{Variant for the Grothendieck resolution}
\label{ss:groth-resolution}

In this subsection we consider analogues of the objects studied above for the Grothendieck resolution $\tgg$. Recall that this variety is the vector bundle on the flag variety $\Flag$ of $\GB$ (considered as the variety of Borel subgroups in $\GB$) defined as
\[
\tgg=\{(x,\BB') \in \gg \times \Flag \mid x \in \Lie(\BB')\}.
\]

There exists a natural projective morphism
$\pi \colon \tgg \to \gg$ defined by $\pi(x,\BB')=x$. We denote by $\tgg_{\reg}$, resp.~$\tgg_\rs$, the inverse image of $\gg_{\reg}$, resp.~$\gg_\rs$, in $\tgg$, and by $\pi_\reg:\tgg_\reg \to \gg_\reg$, resp.~$\pi_\rs \colon \tgg_\rs \to \gg_\rs$, the restriction of $\pi$. 

We will also use the morphism $\nu \colon \tgg \to \tg$ defined as follows: for $g \in \GB$ and $x \in \Lie(g\BB g^{-1})$, $\nu(x, g \BB g^{-1})$ is the inverse image under the isomorphism $\tg \hookrightarrow \bg \twoheadrightarrow \bg/\ng$ of the image of $g^{-1} \cdot x$ in $\bg/\ng$. (One can easily check that this morphism is well defined, and is a smooth morphism which satisfies $\tgg_\rs=\nu^{-1}(\tg_\rs)$, see e.g.~\cite[\S 13.3]{jantzen}.) We denote by $\nu_\reg \colon \tgg_\reg \to \tg$ the restriction of $\nu$ to $\tgg_\reg$. Combining $\pi$ and $\nu$ we obtain a morphism
\[
\varphi \colon \tgg \to \gg \times_{\tg/W} \tg.
\]
We denote by $\varphi_\reg \colon \tgg_{\reg} \to \ggr \times_{\tg /W} \tg$, resp.~$\varphi_{\mathrm{rs}} \colon \tgg_{\mathrm{rs}} \to \gg_{\mathrm{rs}} \times_{\tg_{\mathrm{rs}} / W} \tg_{\mathrm{rs}}$, the restriction of $\varphi$ to $\tgg_\reg$, resp.~$\tgg_{\mathrm{rs}}$.

We will consider
the schemes
\[
\tSC := \SC \times_{\gg} \tgg, \qquad \tUp := \Upsilon \times_{\gg} \tgg.
\]
These schemes are also endowed with an action of $\Gm$, obtained by restricting the action on $\tgg$ defined by
\[
t \cdot (x,\BB') :=(t^{-2} {\check \lambda}_\circ(t) \cdot x, {\check \lambda}_\circ(t) \BB' {\check \lambda}_\circ(t)^{-1})
\]
for $t \in \Gm$ and $(x,\BB') \in \tgg$. The natural morphisms $\tSC \to \SC$ and $\tUp \to \Upsilon$ are $\Gm$-equivariant. Note that the set-theoretic fiber of $\pi$ over $e$ is reduced to $(e,\BB^+)$ by~\cite[Lemma~5.3]{springer}. Hence this $\Gm$-action on $\tUp$ is contracting to $(e,\BB^+)$.

One can also consider the universal stabilizer $\tIB$ over $\tgg$, defined as the fiber product
\[
\tIB:=\tgg \times_{\tgg \times \tgg} (\GB \times \tgg).
\]
If $(x,\BB')$ is a point in $\tgg$, then the fiber of $\tIB$ over $(x,\BB')$ identifies with $(\BB')_x$. We will denote by $\tIB_{\reg}$ the restriction of $\tIB$ to $\tgg_{\reg}$, by $\tIB_{\rs}$ its restriction to $\tgg_{\rs}$, and by $\tIB_{\SC}$ its restriction to $\tSC$. If
{\rm (C3)} holds, then it follows from Corollary~\ref{cor:Ireg-commutative} that these group schemes are commutative.

\begin{lem}
\label{lem:ta-smooth}
Assume that {\rm (C3)} holds. Then
the morphism
\[
\widetilde{a} \colon \GB \times \tSC \to \tgg_\reg
\]
induced by the $\GB$-action on $\tgg$ is smooth and surjective.
\end{lem}

\begin{proof}
By $\GB$-equivariance the following diagram is Cartesian:
\[
\xymatrix@C=1.5cm@R=0.5cm{
\GB \times \tSC \ar[r]^-{\widetilde{a}} \ar[d] & \tgg_\reg \ar[d]^-{\pi_\reg} \\
\GB \times \SC \ar[r]^-{a} & \gg_\reg,
}
\]
where the left vertical map is the product of $\mathrm{id}_\GB$ with the morphism induced by $\pi$. Hence smoothness and surjectivity of $\widetilde{a}$ follow from the same properties for $a$, which we proved in Lemma~\ref{lem:a-smooth}.
\end{proof}

\begin{remark}
Using similar arguments one can prove that
the morphism $\UB \times \tSC \to \tUp$ induced by the $\UB$-action on $\tgg$ is an isomorphism, assuming that (C2) holds.
\end{remark}

\begin{lem}
\label{lem:tgg-reg}
Assume that {\rm (C3)} holds. Then
the morphism 
\[
\varphi_\reg \colon \tgg_{\reg} \to \ggr \times_{\tg /W} \tg
\]
is an isomorphism.
\end{lem}

\begin{proof}
Following the arguments in~\cite[Remark~4.2.4(i)]{ginzburg}, we only have to prove the following properties:
\begin{enumerate}
\item the morphism $\varphi_{\reg}$ is finite;
\item the scheme $\ggr \times_{\tg /W} \tg$ is a smooth variety;
\item $\varphi_{\mathrm{rs}}$ is an isomorphism.
\end{enumerate} 
Indeed, then the claim follows from the general result that if $f \colon X \to Y$ is a finite, birational morphism of integral Noetherian schemes such that $Y$ is normal, then $f$ is an isomorphism.

First, we claim that the projection $\pi_{\reg} \colon \tgg_{\reg} \to \gg_{\reg}$ is quasi-finite. In fact,  since the locus where $\pi_\reg$ is quasi-finite is open (see~\cite[Tag 01TI]{stacks-project}), it is enough to prove that this morphism is quasi-finite at all closed points of $\tgg_\reg$. Since any closed point in $\tgg_\reg$ is $\GB$-conjugate to a closed point in $\tSC$ (see Lemma~\ref{lem:ta-smooth}), it is enough to prove that $\pi_\reg$ is quasi-finite at all closed point of $\tSC$. Using the contracting $\Gm$-action on $\tSC$ (and again the fact that the quasi-finite locus is open), it is enough to prove that $\pi_\reg$ is quasi-finite at $(e,\BB^+)$. This property follows from~\cite[Tag 02NG]{stacks-project} and the fact that the 
set-theoretic fiber of $\pi_\reg$ over $e$ is reduced to $(e,\BB^+)$.

Since $\pi_{\reg}$ is quasi-finite, a fortiori $\varphi_{\mathrm{reg}}$ is quasi-finite, see~\cite[Tag 03WR]{stacks-project}. Since this morphism is also projective, it is finite, see~\cite[Tags 02NH \& 02LS]{stacks-project}, which proves $(1)$.

Property~$(2)$ is a consequence of Proposition~\ref{prop:chi-reg-smooth}.

Finally we turn to $(3)$. We remark that $\pi_{\rs}$ is {\'e}tale by~\cite[Lemma~13.4]{jantzen}.
On the other hand, the quotient morphism $\tg_\rs \to \tg_\rs / W$ is {\'e}tale by Lemma~\ref{lem:centralizer-ss-reg}\eqref{it:W-trs-free} and~\cite[Remark~12.8]{jantzen}, hence the morphism $\gg_\rs \times_{\tg_\rs/W} \tg_\rs \to \gg_\rs$ is also {\'e}tale. We deduce that $\varphi_\rs$ is {\'e}tale. Moreover, comparing the fibers of $\pi_\rs$ (see~\cite[Lemma~13.3]{jantzen}) and of the projection $\gg_\rs \times_{\tg_\rs/W} \tg_\rs \to \gg_\rs$ (see Lemma~\ref{lem:centralizer-ss-reg}\eqref{it:W-trs-free}) over closed points, we obtain that $\varphi_\rs$ induces a bijection at the level of closed points. Using~\cite[Tag 04DH]{stacks-project}, we deduce that $\varphi_\rs$ is an isomorphism.
\end{proof}

\begin{remark}
\label{rk:tggreg-parabolic}
Let $\PB \subset \GB$ be a standard parabolic subgroup of $\GB$, let $W_\PB \subset W$ be the Weyl group of the Levi factor of $\PB$ containing $\TB$, and consider the variety $\tgg^\PB:=\{(x,g\PB) \in \gg \times \GB/\PB \mid x \in \Lie(g \PB g^{-1})\}$. There exist natural morphisms $\tgg \to \tgg^\PB \to \gg$, and we denote by $\tgg^\PB_\reg \subset \tgg^\PB$ the inverse image of $\gg_\reg$. If (C3) holds, there exist canonical isomorphisms
\[
\tgg^\PB_\reg \simto \gg_\reg \times_{\tg/W} (\tg/W_\PB) \qquad \text{and} \qquad \tgg_\reg \simto \tgg^\PB_\reg \times_{\tg / W_\PB} \tg.
\]
(Indeed, the same arguments as for Lemma~\ref{lem:tgg-reg} prove the first isomorphism, and the second one follows.)
\end{remark}

\begin{prop}
\label{prop:tilde-S-Up}
Assume that {\rm (C3)} holds.

\begin{enumerate}
\item
The morphisms
\[
\tSC \to \SC \times_{\tg/W} \tg, \quad \text{and} \quad \tUp \to \Upsilon \times_{\tg/W} \tg
\]
induced by $\varphi$ are isomorphisms. In particular,
$\tSC$ and $\tUp$ are affine schemes.
\item
The morphism $\nu_\reg$ restricts to an isomorphism $\tSC \simto \tg$.
\label{it:tSC-t}
\end{enumerate}
\end{prop}

\begin{proof}
(1) follows from~\eqref{eqn:Up-reg} and Lemma~\ref{lem:tgg-reg}. Then 
(2) follows from (1) and Theorem~\ref{thm:kostant}.
\end{proof}

Now we consider universal centralizers.

\begin{prop}
\label{prop:tIreg-smooth}
Assume that {\rm (C3)} holds. The natural commutative diagram
\[
\xymatrix@C=1.5cm@R=0.5cm{
\tIB_\reg \ar[r] \ar[d] & \IB_\reg \ar[d] \\
\tgg_\reg \ar[r]^-{\pi_\reg} & \gg_\reg
}
\]
is Cartesian.
In particular,
the group scheme $\tIB_\reg$ is smooth over $\tgg_\reg$.
\end{prop}

\begin{proof}
The first claim follows from the compatibility of universal centralizers with fiber products. More specifically, by definition we have
\[
\tIB_\reg = \tgg_\reg \times_{\tgg_\reg \times \tgg_\reg} (\GB \times \tgg_\reg).
\]
Since both morphisms in this fiber product factor through $\tgg_\reg \times_\tg \tgg_\reg$, we deduce that
\[
\tIB_\reg \cong \tgg_\reg \times_{\tgg_\reg \times_\tg \tgg_\reg} (\GB \times \tgg_\reg) \cong \tgg_\reg \times_{\tgg_\reg \times_\tg \tgg_\reg} ((\GB \times \tg) \times_\tg \tgg_\reg).
\]
Now by Lemma~\ref{lem:tgg-reg} the right-hand side is naturally isomorphic to
\[
\tg \times_{\tg/W} \bigl( \gg_\reg \times_{\gg_\reg \times_{\tg/W} \gg_\reg} ((\GB \times \tg/W) \times_{\tg/W} \gg_\reg) \bigr) \cong \tg \times_{\tg/W} \IB_\reg \cong \tgg_\reg \times_{\gg_\reg} \IB_\reg.
\]
The second claim follows from the first one and Corollary~\ref{cor:Ireg-smooth}.
\end{proof}

\begin{remark}
\begin{enumerate}
\item
Consider the setting of Remark~\ref{rk:tggreg-parabolic}. If $\tIB_\reg^\PB$ denotes the restriction of the universal centralizer over $\tgg^\PB$ to $\tgg^\PB_\reg$, then the same arguments as for Proposition~\ref{prop:tIreg-smooth} show that the following diagrams are Cartesian:
\[
\xymatrix@R=0.5cm{
\tIB_\reg \ar[r] \ar[d] & \tIB^\PB_\reg \ar[d] \\
\tgg_\reg \ar[r] & \tgg_\reg^\PB,
} \qquad \qquad
\xymatrix@R=0.5cm{
\tIB_\reg^\PB \ar[r] \ar[d] & \IB_\reg \ar[d] \\
\tgg_\reg^\PB \ar[r] & \gg_\reg.
}
\]
\item
Assume that (C3) holds. Then
Proposition~\ref{prop:tIreg-smooth} implies that for $(x,\BB') \in \tgg_\reg$ the inclusion $(\BB')_x \hookrightarrow \GB_x$ is an equality.
\end{enumerate}
\end{remark}

Let us record the following immediate consequence of Proposition~\ref{prop:tIreg-smooth}, for later reference.

\begin{cor}
\label{cor:tI-smooth}
Assume that {\rm (C3)} holds.
The natural commutative diagram
\[
\xymatrix@R=0.5cm{
\tIB_\SC \ar[r] \ar[d] & \IB_\SC \ar[d] \\
\tSC \ar[r] & \SC
}
\]
is Cartesian.
In particular, the group scheme $\tIB_\SC$ is smooth over $\tSC$.
\end{cor}

From Proposition~\ref{prop:J} and Proposition~\ref{prop:tIreg-smooth} we deduce that, if (C3) holds, there exists a smooth, affine and commutative group scheme $\tJB$ on $\tg$ whose pullback under $\nu_\reg \colon \tgg_\reg \to \tg$ is $\tIB_\reg$. In fact we have Cartesian diagrams
\begin{equation}
\label{eqn:cartesian-tJ-J}
\vcenter{
\xymatrix@R=0.5cm{
\tJB \ar[r] \ar[d] & \JB \ar[d] \\
\tg \ar[r] & \tg/W
}
}
\qquad \text{and} \qquad
\vcenter{
\xymatrix@R=0.5cm{
\tIB_\SC \ar[r]^-{\sim} \ar[d] & \tJB \ar[d] \\
\tSC \ar[r]^-{\sim} & \tg.
}
}
\end{equation}

\begin{remark}
Assume that (C3) holds.
Using similar constructions as in Proposition~\ref{prop:coh-greg}, one can construct a ``Kostant--Whittaker reduction'' equivalence of categories $\Coh^\GB(\tgg_\reg) \simto \Rep(\tJB)$.
\end{remark}

We set
\[
\tIG:=\Lie(\tIB/\tgg), \ \ \tIG_\reg:=\Lie(\tIB_\reg/\tgg_\reg), \ \ 
\tIG_{\SC}:=\Lie(\tIB_\SC/\tSC), \ \ \tJG:=\Lie(\tJB/\tg)
\]
and
\[
\tIM:=\LLie(\tIB/\tgg), \ \ \tIM_\reg:=\LLie(\tIB_\reg/\tgg_\reg), \ \
\tIM_{\SC}:=\LLie(\tIB_\SC/\tSC), \ \ \tJM:=\LLie(\tJB/\tg).
\]

The following proposition is analogous to Proposition~\ref{prop:properties-Lie-alg}, and its proof is similar (hence left to the reader).

\begin{prop}
\label{prop:properties-Lie-alg-tilde}
Assume that {\rm (C3)} holds. Then:
\begin{enumerate}
\item
the coherent sheaf $\tIG_\reg$, 
resp.~$\tIG_\SC$, resp.~$\tJG$ on $\tgg_\reg$, 
resp.~$\tSC$, resp.~$\tg$, is locally free of finite rank;
\item
the Lie algebra $\tIM_\reg$, 
resp.~$\tIM_\SC$, resp.~$\tJM$, is a vector bundle over $\tgg_\reg$, 
resp.~$\tSC$, resp.~$\tg$;
\item
the Lie algebras $\tIG_\reg$, 
$\tIG_\SC$ and $\tJG$ are commutative;
\item
the restriction of $\tIG_\reg$ to $\tSC$ is canonically isomorphic to $\tIG_\SC$; in other words there exists a canonical Cartesian diagram
\[
\xymatrix@R=0.5cm{
\tIM_\SC \ar[r] \ar[d] & \tIM_\reg \ar[d] \\
\tSC \ar[r] & \tgg_\reg;
}
\]
\item
there exists a canonical isomorphism $\nu_\reg^*(\tJG) \cong \tIG_\reg$; in other words there exists a canonical Cartesian diagram
\[
\xymatrix@R=0.5cm{
\tIM_\reg \ar[r] \ar[d] & \tJM \ar[d] \\
\tgg_\reg \ar[r]^-{\nu_\reg} & \tg;
}
\]
\item
the coherent sheaf $\tIG_\SC$ is canonically isomorphic to the pullback of $\tJG$ under the isomorphism $\tSC \simto \tg$; in other words $\tIM_\SC$ is canonically isomorphic to the pullback of $\tJM$ under the isomorphism $\tSC \simto \tg$.\qed
\end{enumerate}
\end{prop}

\begin{lem}
\label{lem:tJG}
Assume that {\rm (C3)} holds. Then
there exist canonical Cartesian diagrams
\[
\vcenter{
\xymatrix@R=0.5cm{
\tIM_\SC \ar[r] \ar[d] & \IM_\SC \ar[d] \\
\tSC \ar[r] & \SC,
}
}
\qquad
\vcenter{
\xymatrix@R=0.5cm{
\tIM_\reg \ar[r] \ar[d] & \IM_\reg \ar[d] \\
\tgg_\reg \ar[r] & \gg_\reg
}
}
\quad \text{and} \quad
\vcenter{
\xymatrix@R=0.5cm{
\tJM \ar[r] \ar[d] & \JM \ar[d] \\
\tg \ar[r] & \tg/W.
}
}
\]
\end{lem}

\begin{proof}
This follows from Lemma~\ref{lem:Lie-alg}\eqref{it:base-change-Lie}, Corollary~\ref{cor:Ireg-smooth}, Proposition~\ref{prop:tIreg-smooth}, Corollary~\ref{cor:tI-smooth}, and the left diagram in~\eqref{eqn:cartesian-tJ-J}.
\end{proof}

Combining Theorem~\ref{thm:Lie-alg-cotangent} and Lemma~\ref{lem:tJG} we obtain the following.

\begin{thm}
\label{thm:Lie-alg-cotangent-tilde}
Assume that {\rm (C4)} holds. Then
there exist canonical isomorphisms
\[
\tJM \cong \tg \times_{\tg/W} \TM^*(\tg/W), \qquad \tIM_\reg \cong \tgg_\reg \times_{\tg/W} \TM^*(\tg/W), \qquad \tIM_{\SC} \cong \tSC \times_{\SC} \TM^*(\SC)
\]
of commutative Lie algebras over $\tg$, $\tgg_\reg$ and $\tSC$, respectively.\qed
\end{thm}

\section{The case of integral coefficients}
\label{sec:R}


\subsection{Reduction to algebraically closed fields}
\label{ss:reduction-fields}

Let $\RG$ be a finite localization of $\ZM$. In this subsection we prove various results that allow to deduce results over $\RG$ from their analogues over algebraically closed fields of positive characteristic. These results (which are independent of the rest of the paper) will be used crucially in the rest of the section.

\begin{lem}
\label{lem:flat-R-F}
Let $A$ be a finitely generated $\RG$-algebra, and let $A'$ be a finitely generated $A$-algebra.
Assume that $A'$ is flat over $\RG$ and that for any geometric point $\FM$ of $\RG$ of positive characteristic,
$\FM \otimes_{\RG} A'$ is flat over $\FM \otimes_\RG A$. Then $A'$ is flat over $A$.
\end{lem}

\begin{proof}
By~\cite[Tag 00HT, Item (7)]{stacks-project}, it suffices to prove that for all maximal ideals $\mg \subset A'$, with $\pg$ the inverse image of $\mg$ in $A$, the $A_\pg$-module $A'_\mg$ is flat. Let also $\qg$ be the inverse image of $\mg$ in $\RG$. Then, by~\cite[Tag 00GB]{stacks-project}, $\qg$ is a maximal ideal of $\RG$, i.e.~is of the form $\ell \cdot \RG$ for some prime number $\ell$ not invertible in $\RG$. From our assumption on geometric points one can easily deduce that $A' / (\ell \cdot A')$ is flat over $A / (\ell \cdot A)$. Hence, using~\cite[Tag 00HT, Item (7)]{stacks-project} again, we deduce that $(A' / (\ell \cdot A'))_{\mg/ \ell \cdot A'} = A'_\mg / \qg \cdot A'_\mg$ is flat over $(A / (\ell \cdot A))_{\pg/ \ell \cdot A} = A_\pg / \qg \cdot A_\pg$. By the same result and our other assumption we know that $A'_\mg$ is flat over $\RG_\qg$. Hence we can apply~\cite[Tag 00MP]{stacks-project} to the morphisms $\RG_\qg \to A_\pg \to A'_\mg$ to deduce that $A'_\mg$ is flat over $A_\pg$, which finishes the proof.
%
%
\end{proof}

In the next statements, if $X$ is an $\RG$-scheme and $\SG$ is an $\RG$-algebra, we set $X_\SG:=X \times_{\Spec(\RG)} \Spec(\SG)$.

\begin{cor}
\label{cor:flat-reduction-fields}
Let $X$ and $Y$ be schemes which are locally of finite type over $\RG$, and let $f \colon X \to Y$ be a morphism. Assume that $X$ is flat over $\RG$ and that
for any geometric point $\FM$ of $\RG$ of positive characteristic,
the morphism $X_\FM \to Y_\FM$ induced by $f$
is flat. Then $f$ is flat.
\end{cor}

\begin{proof}
Since flatness is a property which is local both on $X$ and $Y$, we can assume that both of them are affine. Then the claim follows from Lemma~\ref{lem:flat-R-F}.
\end{proof}

\begin{prop}
\label{prop:smooth-reduction-fields}
Let $X$ and $Y$ be schemes which are 
of finite type over $\RG$, and let $f \colon X \to Y$ be a morphism. Assume that
$X$ is flat over $\RG$ and that
for any geometric point $\FM$ of $\RG$ of positive characteristic,
the morphism $X_\FM \to Y_\FM$ induced by $f$
is smooth. Then $f$ is smooth.
\end{prop}

\begin{proof}
First, by Corollary~\ref{cor:flat-reduction-fields}, $f$ is flat. 

The smooth locus of $f$ is open (see e.g.~the definition of smoothness in~\cite[Tag 01V5]{stacks-project}) and the closed points are dense in every closed subset of $X$ (see~\cite[Tag 00G3]{stacks-project}), hence it is enough to prove that $f$ is smooth at closed points of $X$. Let $x$ be such a closed point.
By~\cite[Tag 00GB]{stacks-project}, the residue field $\kappa(x)$ has characteristic $\ell$ for some prime number $\ell$ not invertible in $\RG$. Let $\FM$ be an algebraic closure of $\kappa(x)$, let $x'$ be the point in $X_{\kappa(x)}$ whose image in $X$ is $x$, and let $x''$ be the unique closed point of $X_{\FM}$ lying over $x'$.
Let also $f' \colon X_{\kappa(x)} \to Y_{\kappa(x)}$ and $f'' \colon X_\FM \to Y_\FM$ be the morphisms induced by $f$. Then $f''$ is smooth by assumption. Moreover, one can see (using e.g.~the characterization of smoothness in~\cite[Tag 01V9, Item (3)]{stacks-project} and the invariance of dimension under field extension, see~\cite[Proposition~5.38]{gw}) that the smoothness of $f''$ at $x''$ is equivalent to the smoothness of $f'$ at $x'$, which is itself equivalent to the smoothness of $f$ at $x$. This finishes the proof.
\end{proof}

\begin{lem}
\label{lem:affine-R-F}
Let $A$ be a finitely generated $\RG$-algebra.
Let $X$ be a Noetherian $A$-scheme which is projective over $A$ and flat over $\RG$. Assume that for any 
geometric point $\FM$ of $\RG$ of positive characteristic,
the scheme $X_\FM$ is affine. Then $X$ is affine.
\end{lem}

\begin{proof}
By Serre's criterion (see~\cite[Theorem~III.3.7]{hartshorne}) it suffices to prove that for any coherent sheaf of ideals $\IC \subset \OC_X$, the complex of $A$-modules $R\Gamma(X,\IC)$ is concentrated in degree $0$. However by Grothendieck's vanishing theorem~\cite[Theorem~III.2.7]{hartshorne} this complex is bounded, and since $X$ is projective over $A$ its cohomology sheaves are finitely generated over $A$ (see~\cite[Theorem~III.5.2]{hartshorne}). Hence by~\cite[Lemma~1.4.1(2)]{br} it suffices to prove that for any 
geometric point $\FM$ of $\RG$,
the complex $\FM \lotimes_{\RG} R\Gamma(X,\IC)$ is concentrated in degree $0$. 

Now one can easily check that we have
\[
\FM \lotimes_{\RG} R\Gamma(X,\IC) \cong R\Gamma(X,\FM \lotimes_{\RG} \IC).
\]
Moreover, since $\OC_X$ is flat over $\RG$, the complex $\FM \lotimes_{\RG} \IC$ is concentrated in degree $0$, and isomorphic to the direct image of a coherent sheaf $\IC_\FM$ on $X_\FM$. Hence we have
\[
\FM \lotimes_{\RG} R\Gamma(X,\IC) \cong R\Gamma(X_\FM, \IC_\FM),
\]
and the desired claim follows from the assumption that $X_\FM$ is affine and the other direction in Serre's criterion.
\end{proof}

\subsection{Definitions and preliminary results}
\label{ss:definitions}

From now on, we
let $\GB_\ZM$ be a group scheme over $\ZM$ which is a product of split simply connected quasi-simple groups, and general linear groups $\mathrm{GL}_{n,\ZM}$. In particular, $\GB_\ZM$ is a split connected reductive group over $\ZM$.
We let $\BB_\ZM \subset \GB_\ZM$ be a Borel subgroup, and $\TB_\ZM \subset \BB_\ZM$ be a (split) maximal torus. We denote by $\gg_\ZM$, $\bg_\ZM$, $\tg_\ZM$ the Lie algebras of $\GB_\ZM$, $\BB_\ZM$, $\TB_\ZM$.

We let $N$ be the product of all the prime numbers which are not very good for some quasi-simple factor of $\GB_\ZM$, and set $\RG:=\ZM[1/N]$. We let $\GB_\RG$, $\BB_\RG$, $\TB_\RG$ be the groups obtained from $\GB_\ZM$, $\BB_\ZM$, $\TB_\ZM$ by base change to $\RG$, and $\gg_\RG$, $\bg_\RG$, $\tg_\RG$ be their respective Lie algebras. We have $\gg_\RG = \RG \otimes_\ZM \gg_\ZM$.
We also denote by $\UB_\RG$ the unipotent radical of $\BB_\RG$, by $\ng_\RG$ its Lie algebras, and by $W$ the Weyl group of $(\GB_\RG, \TB_\RG)$.

We denote by $\Phi \subset X^*(\TB_\RG)$ the root system of $\GB_\RG$, by $\Phi^+ \subset \Phi$ the set of roots which are opposite to the $\TB_\RG$-weights in $\ng_\RG$, by $\Delta \subset \Phi^+$ the corresponding basis, by ${\check \Phi} \subset X_*(\TB_\RG)$ the coroots, and by ${\check \Phi}^+ \subset {\check \Phi}$ the coroots corresponding to $\Phi^+$.

For any geometric point $\FM$ of $\RG$ we set
\[
\GB_\FM := \Spec(\FM) \times_{\Spec(\RG)} \GB_\RG.
\]
We also denote by $\BB_\FM$, $\TB_\FM$ the base change of $\BB_\RG$, $\TB_\RG$, and by $\tg_\FM \subset \bg_\FM \subset \gg_\FM$ the Lie algebras of $\TB_\FM \subset \BB_\FM \subset \GB_\FM$.

\begin{prop}
\label{prop:Chevalley-thm-R}
The inclusion $\tg_\RG \hookrightarrow \gg_\RG$ induces an isomorphism of $\RG$-schemes
\[
\tg_\RG / W \simto
\gg_\RG / \GB_\RG.
\]
Moreover, these $\RG$-schemes are smooth (in fact they are affine spaces).
\end{prop}

\begin{proof}
Consider the first claim.
Clearly, it is enough to prove the claim in the case $\GB_\RG$ is either simply-connected and quasi-simple, or else equal to $\mathrm{GL}_{n,\ZM}$. The first case is treated in~\cite[Theorem~1]{chaputromagny}. In the second case, we have $\RG=\ZM$, and one can argue as follows. We can assume that $\TB_\ZM$ is the standard maximal torus consisting of diagonal matrices. Considering the natural diagram
\[
\xymatrix{
\ZM[\mathfrak{gl}_{n,\ZM}]^{\mathrm{GL}_{n,\ZM}} \ar[r] \ar[d] & \ZM[\tg_\ZM]^W \ar[d] \\
\CM[\mathfrak{gl}_{n,\CM}]^{\mathrm{GL}_{n,\CM}} \ar[r] & \CM[\tg_\CM]^W,
}
\]
in which the vertical morphisms are injective and the lower horizontal morphism is an isomorphism (by Proposition~\ref{prop:Chevalley-thm}), we obtain the injectivity of the upper horizontal arrow. To prove surjectivity we only have to prove that the elementary symmetric functions $e_i \in \ZM[\tg_\ZM]=\ZM[X_1, \cdots, X_n]$ ($1 \leq i \leq n$) belong to the image; however, $e_i$ is the image of the function on $\mathfrak{gl}_{n,\ZM}$ sending a matrix $M$ to the trace of the $i$-th exterior power of $M$, which clearly belongs to $\ZM[\mathfrak{gl}_{n,\ZM}]^{\mathrm{GL}_{n,\ZM}}$.

The second claim follows from~\cite[Th{\'e}or{\`e}me 3]{demazure}.
\end{proof}

Using Proposition~\ref{prop:Chevalley-thm-R}, we will freely identify $\gg_\RG / \GB_\RG$ with $\tg_\RG / W$. Let us also note the following corollary.

\begin{cor}
\label{cor:base-change-quotient}
For any geometric point $\FM$ of $\RG$, the natural morphism
\[
\gg_\FM / \GB_\FM \to \Spec(\FM) \times_{\Spec(\RG)} (\gg_\RG / \GB_\RG)
\]
is an isomorphism.
\end{cor}

\begin{proof}
By Proposition~\ref{prop:Chevalley-thm} and Proposition~\ref{prop:Chevalley-thm-R} it is enough to prove that the natural morphism
\[
\tg_\FM / W \to \Spec(\FM) \times_{\Spec(\RG)} (\tg_\RG/W)
\]
is an isomorphism. The latter property follows from~\cite[Corollary on p.~296]{demazure}.
\end{proof}

\begin{lem}
\label{lem:bilinear-form}
There exists a symmetric $\GB_\RG$-invariant bilinear form on $\gg_\RG$ which is a perfect pairing.
\end{lem}

\begin{proof}
As in the proof of Proposition~\ref{prop:Chevalley-thm-R}, we can assume that $\GB_\RG$ is either quasi-simple and simply-connected, or isomorphic to $\mathrm{GL}_{n,\ZM}$. The second case is easy. In the first case, if $\GB_\RG$ is exceptional then by~\cite[Corollary~I.4.9]{springer-steinberg} the discriminant of the Killing form of $\gg_\ZM$ is invertible in $\RG$, hence the Killing form of $\gg_\RG$ satisfies the conditions of the lemma. If $\GB_\RG$ is a classical simply-connected quasi-simple group, then $\gg_\RG$ can be realized naturally as a subalgebra of $\mathfrak{gl}(m,\RG)$ for some $m$, and one can check that the restriction of the bilinear form 
\begin{equation}
\label{eqn:trace-form}
\left\{
\begin{array}{ccc}
\mathfrak{gl}(m,\RG) \otimes_\RG \mathfrak{gl}(m,\RG) & \to & \RG \\
X \otimes Y & \mapsto & \mathrm{tr}(XY)
\end{array}
\right.
\end{equation}
is non-degenerate; see the considerations in~\cite[p.~184]{springer-steinberg}. For instance, let us explain this argument in types $B$ or $D$: in this case there exists a natural quotient morphism 
\[
\GB_\RG \to \mathrm{SO}_{m,\RG} = \{M \in \mathrm{GL}_{m,\RG} \mid {}^t \hspace{-1pt} M \cdot M = 1 \}
\]
for some $m \geq 3$. The induced morphism on Lie algebras is an isomorphism since $2$ is invertible in $\RG$, so that we can identify $\gg_\RG$ with 
\[
\mathfrak{so}(m,\RG) = \{x \in \mathfrak{gl}(m,\RG) \mid {}^t \hspace{-1pt} x = -x\}.
\]
Then the submodule
\[
\{x \in \mathfrak{gl}(m,\RG) \mid {}^t \hspace{-1pt} x = x \}
\]
is a complement to $\mathfrak{so}(m,\RG)$ in $\mathfrak{gl}(m,\RG)$ which is orthogonal to $\mathfrak{so}(m,\RG)$ for~\eqref{eqn:trace-form}, which proves that the restriction to  $\mathfrak{so}(m,\RG)$ is a perfect pairing.
\end{proof}

Lemma~\ref{lem:bilinear-form} implies in particular that for any geometric point $\FM$ of $\RG$, the group 
$\GB_\FM$
satisfies condition (C4) of~\S\ref{ss:notation}. It follows that all the results of Section~\ref{sec:fields} are applicable to this group.

From now on, we fix a bilinear form as in Lemma~\ref{lem:bilinear-form}, and denote by $\kappa \colon \gg_\RG \simto \gg^*_\RG$ the induced isomorphism of $\GB_\RG$-modules. (Here $\gg^*_\RG:=\Hom_\RG(\gg_\RG, \RG)$.)

\subsection{Kostant section}
\label{ss:Kostant-section-R}


%

We now explain how to define the Kostant section over $\RG$. For any $\alpha \in \Delta$ we choose a vector $e_\alpha \in \gg_\ZM$ which forms a $\ZM$-basis of the $\alpha$-weight space in $\gg_\ZM$ (with respect to the action of $\TB_\ZM$), and set
\[
e=\sum_{\alpha \in \Delta} e_\alpha.
\]
We denote similarly the image of this vector in $\gg_\RG$.

\begin{lem}
\label{lem:springer-R}
The $\RG$-module $\bg_\RG / [e, \ng_\RG]$ is free of rank $r$.
\end{lem}

\begin{proof}
This follows from the results of~\cite{springer} as in the proof of Lemma~\ref{lem:springer}.
\end{proof}

Now we consider the cocharacter ${\check \lambda}_\circ:=\sum_{{\check \alpha} \in {\check \Phi}^+} {\check \alpha} \in X_*(\TB_\RG)$, and the $(\Gm)_\RG$-action on $\gg_\RG$ defined by
\[
t \cdot x := t^{-2} {\check \lambda}_\circ(t) \cdot x.
\]
Then $e$ is fixed under this action, and the subalgebras $\bg_\RG$ and $\ng_\RG$ are $(\Gm)_\RG$-stable. Using Lemma~\ref{lem:springer-R} we can choose a $(\Gm)_\RG$-stable free $\RG$-submodule $\sg_\RG \subset \bg_\RG$ such that $\bg_\RG = \sg_\RG \oplus [e,\ng_\RG]$. Then we set
\[
\SC_\RG := e + \sg_\RG, \qquad \Upsilon_\RG := e + \bg_\RG.
\]
It is clear that $\SC_\RG$ and $\Upsilon_\RG$ are $(\Gm)_\RG$-stable.

By construction, if $\FM$ is a geometric point of $\RG$, the base change of $\Upsilon_\RG$ to $\FM$ is the scheme $\Upsilon$ studied in Section~\ref{sec:fields} for the group $\GB_\FM$, and the base change of $\SC_\RG$ is the scheme $\SC$ studied in Section~\ref{sec:fields} for $\GB_\FM$, for the choice $\sg=\FM \otimes_\RG \sg_\RG$.

\begin{prop}
\label{prop:ganginzburg-R}
The morphism
\[
\UB_\RG \times_{\Spec(\RG)} \SC_\RG \to \Upsilon_\RG
\]
induced by the adjoint action is an isomorphism of $\RG$-schemes.
\end{prop}

\begin{proof}
We have to prove that the induced morphism $\OC(\Upsilon_\RG) \to \OC(\UB_\RG \times \SC_\RG)$ is an isomorphism of $\RG$-modules. Now, as in the proof of Proposition~\ref{prop:ganginzburg}, each of these modules is endowed with a natural $(\Gm)_\RG$-action (equivalently, a $\ZM$-grading), and one can easily check that the weight spaces are finitely generated free $\RG$-modules. Moreover, our morphism is $(\Gm)_\RG$-equivariant. Hence it suffices to prove that for any prime $\ell$ not invertible in $\RG$, the induced morphism $\FM_\ell \otimes_\RG \OC(\Upsilon_\RG) \to \FM_\ell \otimes_\RG \OC(\UB_\RG \times \SC_\RG)$ is an isomorphism. The latter result follows from the similar claim for an algebraic closure of $\FM_\ell$, which is a consequence of Proposition~\ref{prop:ganginzburg}.
\end{proof}

\begin{thm}[Kostant's theorem over $\RG$]
\label{thm:kostant-R}
The natural morphisms
\[
\SC_\RG \to \Upsilon_\RG / \UB_\RG \to \gg_\RG / \GB_\RG
\]
are isomorphisms of $\RG$-schemes.
\end{thm}

\begin{proof}
The fact that the morphism $\SC_\RG \to \Upsilon_\RG / \UB_\RG$ is an isomorphism is a consequence of Proposition~\ref{prop:ganginzburg-R}. Hence what remains is to prove that the morphism $\SC_\RG \to \gg_\RG / \GB_\RG$ is an isomorphism. By the same arguments as in the proof of Proposition~\ref{prop:ganginzburg-R} it suffices to observe that for any 
geometric point $\FM$ of $\RG$,
the induced morphism
\[
\Spec(\FM) \times_{\Spec(\RG)} \SC_\RG \to \Spec(\FM) \times_{\Spec(\RG)} (\gg_\RG / \GB_\RG)
\]
is an isomorphism (by Corollary~\ref{cor:base-change-quotient} and
Theorem~\ref{thm:kostant} applied to $\GB_\FM$).
\end{proof}

\begin{prop}
\label{prop:action-smooth-R}
The morphism
\[
a \colon \GB_\RG \times_{\Spec(\RG)} \SC_\RG \to \gg_\RG
\]
induced by the $\GB_\RG$-action on $\gg_\RG$
is smooth.
\end{prop}

\begin{proof}
This follows from Proposition~\ref{prop:smooth-reduction-fields}
and Lemma~\ref{lem:a-smooth}.
\end{proof}

\begin{remark}
\label{rmk:greg-R}
Since $a$ is in particular a flat morphism, its image is an open subscheme in $\gg_\RG$. One can define the regular locus $\gg_\RG^\reg \subset \gg_\RG$ as this image. Then the morphism $a \colon \GB_\RG \times \SC_\RG \to \gg_\RG^\reg$ is smooth and surjective. Moreover, the same argument as in the proof of Proposition~\ref{prop:chi-reg-smooth} shows that the restriction $\gg_\RG^\reg \to \tg_\RG/W$ of the adjoint quotient is smooth and surjective.
\end{remark}

\subsection{The universal centralizer and its Lie algebra}
\label{ss:centralizer-R}
 
We define the (affine) group scheme $\IB_\RG$ over $\gg_\RG$ as the fiber product
\[
\IB_\RG:=\gg_\RG \times_{\gg_\RG \times \gg_\RG} (\GB_\RG \times \gg_\RG),
\]
where the morphisms are similar to those considered in~\S\ref{ss:centralizer-F}. We also denote by $\IB_\SC^\RG$ the restriction of $\IB_\RG$ to $\SC_\RG$.
It is clear that for any 
geometric point $\FM$ of $\RG$,
the base change of $\IB_\RG$, resp.~$\IB_{\SC}^\RG$, to $\FM$ is the corresponding group scheme defined and studied in~\S\ref{ss:centralizer-F} for the group $\GB_\FM$. 

\begin{prop}
\label{prop:IS-smooth-R}
The group scheme $\IB_\SC^\RG$ is smooth over $\SC_\RG$, and commutative.
\end{prop}

\begin{proof}
Smoothness follows from the same arguments as in the proof of Proposition~\ref{prop:IS-smooth}, using Theorem~\ref{thm:kostant-R} and Proposition~\ref{prop:action-smooth-R}.

To prove commutativity, we have to prove that the comultiplication morphism $\OC(\IB_\SC^\RG) \to \OC(\IB_\SC^\RG) \otimes_{\OC(\SC_\RG)} \OC(\IB_\SC^\RG)$ is cocommutative. However using flatness of $\IB_\SC^\RG$ over $\SC_\RG$ (hence over $\RG$) we obtain that the vertical arrows in the natural commutative diagram
\[
\xymatrix@C=1.2cm@R=0.5cm{
\OC(\IB_\SC^\RG) \ar[r] \ar[d] & \OC(\IB_\SC^\RG) \otimes_{\OC(\SC_\RG)} \OC(\IB_\SC^\RG) \ar[d] \\
\OC(\IB_\SC^\CM) \ar[r] & \OC(\IB_\SC^\CM) \otimes_{\OC(\SC_\CM)} \OC(\IB_\SC^\CM)
}
\]
are injective, where $\IB_\SC^\CM:=\Spec(\CM) \times_{\Spec(\RG)} \IB_\SC^\RG$ and $\SC_\CM:=\Spec(\CM) \times_{\Spec(\RG)} \SC_\RG$. Hence the desired cocommutativity follows from the case $\FM=\CM$ of Corollary~\ref{cor:Ireg-commutative}.
\end{proof}

\begin{remark}
\label{rmk:greg-R-2}
If one defines $\gg_\RG^\reg$ as in Remark~\ref{rmk:greg-R}, then the same arguments as in the proof of Corollary~\ref{cor:Ireg-smooth} imply that the restriction of $\IB_\RG$ to the open subscheme $\gg_\RG^\reg$ is smooth (and commutative). Moreover, one can construct an affine smooth group scheme $\JB_\RG$ over $\tg_\RG/W$ as in Proposition~\ref{prop:J}, and an equivalence as in Proposition~\ref{prop:coh-greg}.
\end{remark}

We set 
\[
\IG^\RG_\SC := \Lie(\IB_\SC^\RG/\SC_\RG), \qquad \IM^\RG_\SC := \LLie(\IB_\SC^\RG/\SC_\RG).
\]

\begin{lem}
\label{lem:J-locally-free-R}
\begin{enumerate}
\item
The coherent sheaf $\IG_\RG^\SC$ on $\SC_\RG$ is locally free of finite rank, and the Lie algebra $\IM_\SC^\RG$ is a vector bundle over $\SC_\RG$.
\item
The Lie algebra $\IG_\SC^\RG$ is commutative.
\item
For any geometric point $\FM$ of $\RG$, the Lie algebra $\Spec(\FM) \times_{\Spec(\RG)} \IM_\SC^\RG$ identifies with the Lie algebra $\IM_\SC$ of \S{\rm \ref{ss:Lie-centralizer-F}} for the group $\GB_\FM$.
\label{it:base-change-IR}
\end{enumerate}
\end{lem}

\begin{proof}
These properties follows from Lemma~\ref{lem:Lie-alg} and Proposition~\ref{prop:IS-smooth-R}.
\end{proof}

\begin{thm}
\label{thm:Lie-alg-cotangent-R}
There exists a canonical isomorphism
\[
\IM_\SC^\RG \simto \TM^*(\SC_\RG)
\]
of commutative Lie algebras over $\SC_\RG$. In other words, if one identifies $\SC_\RG$ with $\tg_\RG/W$ via the isomorphism of Theorem~{\rm \ref{thm:kostant-R}}, there exists a canonical isomorphism 
$\IM_\SC^\RG \simto \TM^*(\tg_\RG/W)$
of commutative Lie algebras
over $\tg_\RG/W$.
\end{thm}

\begin{proof}
As in the proof of Theorem~\ref{thm:Lie-alg-cotangent} one can construct a canonical morphism 
$\IG_\SC^\RG \to \Omega_{\SC_\RG}$ of coherent sheaves on $\SC_\RG$.
To prove that this morphism is an isomorphism it suffices to prove that its cone is isomorphic to $0$. Since both $\IG_\SC^\RG$ and $\Omega_{\SC_\RG}$ are flat over $\RG$, to prove this
it suffices to prove that for any 
geometric point $\FM$ of $\RG$,
the induced morphism $\FM \otimes_\RG \IG_\SC^\RG \to \FM \otimes_\RG \Omega_{\SC_\RG}$ is an isomorphism (see~\cite[Lemma~1.4.1(1)]{br}). However, by Lemma~\ref{lem:J-locally-free-R}\eqref{it:base-change-IR}, $\FM \otimes_\RG \IG_\SC^\RG$ is the Lie algebra $\IG_\SC$ of \S\ref{ss:Lie-centralizer-F} for the group $\GB_\FM$, hence the desired claim follows from Theorem~\ref{thm:Lie-alg-cotangent}.
\end{proof}


\subsection{Variant for the Grothendieck resolution}
\label{ss:grothendieck-R}

In this subsection we consider the Grothendieck resolution
\[
\tgg_\RG:=\GB_\RG \times^{\BB_\RG} \bg_\RG.
\]
It is clear that, for any geometric point $\FM$ of $\RG$, the base change of $\tgg_\RG$ to $\FM$ is the variety $\tgg$ studied in~\S\ref{ss:groth-resolution}, for the group $\GB_\FM$.
There is a natural projective morphism $\pi \colon \tgg_\RG \to \gg_\RG$ induced by the adjoint action, and we set
\[
\tSC_\RG := \SC_\RG \times_{\gg_\RG} \tgg_\RG.
\]
We will also consider the natural morphism $\nu \colon \tgg_\RG \to \tg_\RG$.

The following lemma follows from the same arguments as for Lemma~\ref{lem:ta-smooth}, using Proposition~\ref{prop:action-smooth-R}.

\begin{lem}
\label{lem:action-smooth-R-t}
The morphism
\[
\widetilde{a} \colon \GB_\RG \times_{\Spec(\RG)} \tSC_\RG \to \tgg_\RG
\]
induced by the $\GB_\RG$-action on $\tgg_\RG$
is smooth.
\end{lem}

\begin{prop}
\label{prop:tSC-affine}
The scheme $\tSC_\RG$ is affine, and
the natural morphisms
\begin{equation}
\label{eqn:morphisms-tSC}
\tSC_\RG \to \SC_\RG \times_{\tg_\RG/W} \tg_\RG \to \tg_\RG
\end{equation}
are isomorphisms.
\end{prop}

\begin{proof}
First we observe that $\tSC_\RG$ is flat over $\RG$. In fact, it follows from Lemma~\ref{lem:action-smooth-R-t} that $\GB_\RG \times \tSC_\RG$ is flat over $\RG$. Since $\RG$ is a direct factor in $\OC(\GB_\RG)$, this implies that $\tSC_\RG$ is also flat over $\RG$.
Then,
since $\tSC_\RG$ is $\RG$-flat and projective over $\SC_\RG$, the first assertion
follows from Proposition~\ref{prop:tilde-S-Up}$(1)$ and Lemma~\ref{lem:affine-R-F}.

The fact that the second morphism in~\eqref{eqn:morphisms-tSC} is an isomorphism follows from Theorem~\ref{thm:kostant-R}. To prove that the first one is also an isomorphism, consider the induced morphism
\begin{equation}
\label{eqn:tSC-affine}
\OC(\SC_\RG \times_{\tg_\RG / W} \tg_\RG) \to \OC(\tSC_\RG).
\end{equation}
Then both sides are finite modules over $\OC(\tg_\RG/W)$ (by \cite[Th{\'e}or{\`e}me 2]{demazure} and~\cite[Theorem~III.5.2]{hartshorne}, respectively), and are $\RG$-flat. Hence, using~\cite[Lemma~1.4.1(1)]{br}, to prove that~\eqref{eqn:tSC-affine} is an isomorphism it suffices to prove that it becomes an isomorphism after applying $\FM \otimes_\RG (-)$ for all geometric points $\FM$ of $\RG$. The latter fact was proved in Proposition~\ref{prop:tilde-S-Up}(1).
\end{proof}

We define the universal stabilizer $\tIB_\RG$ over $\tgg_\RG$ as the fiber product
\[
\tIB_\RG := \tgg_\RG \times_{\tgg_\RG \times \tgg_\RG} (\GB_\RG \times \tgg_\RG).
\]
We denote by $\tIB_\SC^\RG$ the restriction of $\tIB_\RG$ to $\tSC_\RG$.

\begin{prop}
\label{prop:tI-smooth-R}
The group scheme $\tIB_\SC^\RG$ is smooth over $\tSC_\RG$, and
the following natural commutative diagram is Cartesian:
\[
\xymatrix@R=0.5cm{
\tIB_\SC^\RG \ar[r] \ar[d] & \IB_\SC^\RG \ar[d] \\
\tSC_\RG \ar[r] & \SC_\RG.
}
\]
\end{prop}

\begin{proof}
The proof of the first claim is similar to the proof of Proposition~\ref{prop:IS-smooth}, namely we use isomorphisms
\[
\tIB_\SC^\RG = \tSC_\RG \times_{\tgg_\RG \times \tSC_\RG} (\GB_\RG \times \tSC_\RG) \cong \tSC_\RG \times_{\tgg_\RG \times_{\tg_\RG} \tSC_\RG} (\GB_\RG \times \tSC_\RG) \cong \tSC_\RG \times_{\tgg_\RG} (\GB_\RG \times \tSC_\RG)
\]
(where the last isomorphism uses Proposition~\ref{prop:tSC-affine}), and then the result follows from Lemma~\ref{lem:action-smooth-R-t}.


To prove the second claim, we observe that the commutative square of the statement induces a canonical morphism of $\tSC_\RG$-group schemes
\begin{equation}
\label{eqn:tAG-fiber-product}
\tIB_\SC^\RG \to \tSC_\RG \times_{\SC_\RG} \IB_\SC^\RG.
\end{equation}
Since both group schemes are closed subschemes of $\GB_\RG \times_{\Spec(\RG)} \tSC_\RG$, and since this morphism is compatible with their inclusion in $\GB_\RG \times_{\Spec(\RG)} \tSC_\RG$, \eqref{eqn:tAG-fiber-product} must be a closed embedding. Since the $\RG$-algebra $\OC(\tSC_\RG \times_{\SC_\RG} \IB_\SC^\RG)$ is of finite type and flat over $\RG$, and since $\tIB_\SC^\RG$ is also $\RG$-flat, by~\cite[Lemma~1.4.1(1)]{br}, to prove that~\eqref{eqn:tAG-fiber-product} is an isomorphism it suffices to prove that,  for any 
geometric point $\FM$ of $\RG$, the morphism 
\[
\Spec(\FM) \times_{\Spec(\RG)} \tIB_\SC^\RG \to \Spec(\FM) \times_{\Spec(\RG)} (\tSC_\RG \times_{\SC_\RG} \IB_\SC^\RG)
\]
is an isomorphism,
which follows from Corollary~\ref{cor:tI-smooth}.
\end{proof}

Now we set
\[
\tIG_\SC^\RG:=\Lie(\tIB_\SC^\RG/\tSC_\RG), \qquad \tIM_\SC^\RG:=\LLie(\tIB_\SC^\RG/\tSC_\RG).
\]

\begin{lem}
\label{lem:tJG-R}
There exists a canonical Cartesian diagram
\[
\xymatrix@R=0.5cm{
\tIM_\SC^\RG \ar[d] \ar[r] & \IM_\SC^\RG \ar[d] \\
\tSC_\RG \ar[r] & \SC_\RG.
}
\]
\end{lem}

\begin{proof}
This follows from Lemma~\ref{lem:Lie-alg}\eqref{it:base-change-Lie}, Proposition~\ref{prop:IS-smooth-R} and Proposition~\ref{prop:tI-smooth-R}.
\end{proof}

Combining Lemma~\ref{lem:tJG-R} and Theorem~\ref{thm:Lie-alg-cotangent-R} we deduce the following result.

\begin{thm}
\label{thm:J-cotangent-R}
There exists a canonical isomorphism
\[
\tIM_\SC^\RG \simto \tSC_\RG \times_{\SC_\RG} \TM^*(\SC_\RG)
\]
of commutative Lie algebras over $\tSC_\RG$. In other words, if one identifies $\tSC_\RG$ with $\tg_\RG$ via the isomorphism of Proposition~{\rm \ref{prop:tSC-affine}}, there exists a canonical isomorphism 
$\tIM_\SC^\RG \simto \tg_\RG \times_{\tg_\RG/W} \TM^*(\tg_\RG/W)$
of commutative Lie algebras
over $\tg_\RG$.\qed
\end{thm}

\section{Application: mixed modular derived Satake equivalence}
\label{sec:derived-Satake}

\subsection{Notation}
\label{ss:notation-Satake}

Let $\GD$ be a connected reductive
complex algebraic group. We will consider the affine Grassmannian
\[
\Gr:=\GD(\mathscr{K})/\GDO,
\]
with its natural structure of complex ind-variety, where $\mathscr{K}:=\CM ( \hspace{-1pt} (z) \hspace{-1pt} )$ and $\mathscr{O}:=\CM [ \hspace{-1pt} [z] \hspace{-1pt} ]$. We also choose a Borel subgroup $\BD \subset \GD$ (considered as the ``negative'' Borel subgroup) and a maximal torus $\TD \subset \GD$, and denote by $\UD$ the unipotent radical of $\BD$ and by $\XM:=X_*(\TD)$ the lattice of cocharacters of $\TD$.
Any $\lambda \in \XM$ defines in a natural way a point $L_\lambda \in \Gr$, and we set $\Gr^\lambda:=\GDO \cdot L_\lambda$. Then, if $\XM^+ \subset \XM$ denotes the subset of dominant weights, we have
\[
\Gr = \bigsqcup_{\lambda \in \XM^+} \Gr^\lambda.
\]

We fix $\ell$ which is either $0$ or a prime number which is very good for $\GD$, and let $\FM$ be an algebraically closed field of characteristic $\ell$.
We consider the $\GDO$-equivariant derived category $\Db_{\GDO}(\Gr, \FM)$ with coefficients in $\FM$ (in the sense of Bernstein--Lunts), and the
the full subcategory $\Perv_{\GDO}(\Gr, \FM)$ of perverse sheaves. There exists a natural convolution product $\star$ on $\Db_{\GDO}(\Gr, \FM)$, which makes it a monoidal category. This bifunctor restricts to a bifunctor
\[
(-) \star (-) \colon \Perv_{\GDO}(\Gr, \FM) \times \Perv_{\GDO}(\Gr, \FM) \to \Perv_{\GDO}(\Gr, \FM),
\]
see~\cite{mv}.
We also have a natural functor
\[
\FF := \HF^\bullet(\Gr, -) \colon \Perv_{\GDO}(\Gr, \FM) \to \Vect(\FM),
\]
which has a natural structure of tensor functor.
The main results of~\cite{mv} show that the $\FM$-group scheme $\GB$ of tensor automorphisms of the functor $\FF$ is a 
connected reductive
algebraic group over $\FM$ which is 
Langlands dual to $\GD$, and that this isomorphism induces an equivalence of tensor categories
\[
\Satake \colon \Perv_{\GDO}(\Gr, \FM) \simto \Rep(\GB),
\]
where the right-hand side is the category of finite dimensional (algebraic) representations of $\GB$. Using the Mirkovi{\'c}--Vilonen ``weight functors,'' this equivalence defines a maximal torus $\TB \subset \GB$ and a canonical identification $X^*(\TB)=\XM$, see e.g.~\cite[\S 6.1]{gr} for a brief reminder of these constructions. 
We denote by $\BB$ the Borel subgroup of $\GB$ containing $\TB$ and whose $\TB$-roots are the $\TD$-coroots of $\BD$, and then 
use the same notation
as in Sections~\ref{sec:preliminaries}--\ref{sec:fields}.

The natural morphism
\[
\HF^\bullet_{\GDO}(\mathrm{pt}; \FM) \cong
\HF^\bullet_{\GD}(\mathrm{pt}; \FM) \to \HF^\bullet_{\TD}(\mathrm{pt}; \FM) \cong \mathrm{S}(\tg) = \OC(\tg^*)
\]
is injective, and identifies the left-hand side with $\OC(\tg^* /W)$, see~\cite[\S 3.2]{mr} for references. In particular this graded ring is concentrated in even degrees, so that
we can consider the category $\Parity_{\GDO}(\Gr, \FM)$ of $\GDO$-equivariant $\FM$-parity complexes on $\Gr$ in the sense of~\cite[Definition~2.4]{jmw}. By~\cite[Theorem~2.4 \& \S 4.1]{jmw}, for any $\lambda \in \XM^+$ there exists a unique indecomposable parity complex $\EC^\lambda$ in $\Parity_{\GDO}(\Gr, \FM)$ supported on $\overline{\Gr^\lambda}$ and whose restriction to $\Gr^\lambda$ is $\underline{\FM}_{\Gr^\lambda}[\dim \Gr^\lambda]$. Moreover, any indecomposable parity complex in $\Parity_{\GDO}(\Gr, \FM)$ is isomorphic to $\EC^\lambda[i]$ for some unique $\lambda \in \XM^+$ and $i \in \ZM$.

By~\cite[Corollary~1.6]{mr}, the parity complexes $\EC^\lambda$ are perverse (see also~\cite{jmw2} for an earlier proof of this fact, under stronger assumptions).
We denote by $\PParity_{\GDO}(\Gr,\FM)$ the full subcategory of $\Parity_{\GDO}(\Gr, \FM)$ consisting of parity complexes which are perverse, in other words of direct sums of objects $\EC^\lambda$ (with no shift). 

By~\cite[Theorem~4.8]{jmw}, the convolution product $\star$ restricts to a bifunctor $\Parity_{\GDO}(\Gr, \FM) \times \Parity_{\GDO}(\Gr, \FM) \to \Parity_{\GDO}(\Gr, \FM)$, hence to a bifunctor $\PParity_{\GDO}(\Gr,\FM) \times \PParity_{\GDO}(\Gr,\FM) \to \PParity_{\GDO}(\Gr,\FM)$. It also follows from the results of~\cite{jmw2} that the equivalence $\Satake$ restricts to an equivalence of categories
\[
\PParity_{\GDO}(\Gr,\FM) \simto \Tilt(\GB),
\]
where the right-hand side is the category of finite dimensional tilting $\GB$-modules; see~\cite[\S 1.5]{mr} for remarks and more precise references.

\subsection{Equivariant cohomology of spherical perverse sheaves}
\label{ss:tensor}

In this subsection we assume that the derived subgroup $\DS \GD$ of $\GD$ is quasi-simple. We will denote by $|\cdot|^2$ the unique $W$-invariant quadratic form on $\ZM \Phi$ such that short roots have length one. (Recall that $|\alpha|^2 \in \{1,2,3\}$ for any $\alpha \in \Phi$.)

As explained in~\cite[\S 5.3]{yz} (following earlier ideas developed in particular by Ginzburg~\cite{ginzburg}), the geometric Satake equivalence determines a canonical regular nilpotent element $e \in \ng^+$ which is a sum of non-zero simple root vectors $e_\alpha \in \gg_\alpha$ (for $\alpha \in \Delta$), see in particular~\cite[Proposition~5.6]{yz}. By definition, for $\FC$ in $\Perv_{\GDO}(\Gr, \FM)$, the action of $e$ on $\FF(\FC) = \HF^\bullet(\Gr, \FC)$ is the cup product with the first Chern class of the determinant line bundle on $\Gr$.

For $\alpha \in \Delta$, we will denote by $e_{-\alpha} \in \gg_{-\alpha}$ the unique element such that $[e_\alpha, e_{-\alpha}] = h_\alpha$, where $h_\alpha \in \tg$ is the differential of the coroot of $\GB$ associated with $\alpha$. We also let $\UB_{-\alpha} \subset \UB$ be the unique closed subgroup with Lie algebra $\gg_{-\alpha}$, and $u_{-\alpha} \colon \FM \simto \UB_{-\alpha}$ be the unique group isomorphism whose differential is $e_{-\alpha}$.

In~\cite{yz} the authors give a description of the equivariant cohomology $\HF^\bullet_{\TD}(\Gr, \FC)$ for any $\FC$ in $\Perv_{\GDO}(\Gr, \FM)$, which we now briefly recall.
%
%
%
%
For $\lambda \in \XM$, following~\cite{mv} we set $\TG_\lambda:= \UD(\mathscr{K}) \cdot L_\lambda$, and denote by
\[
t_\lambda \colon \TG_\lambda \hookrightarrow \Gr \quad \text{and} \quad \overline{t}_\lambda \colon \overline{\TG_\lambda} \hookrightarrow \Gr
\]
the inclusions. By~\cite[Proposition~3.1]{mv} we have $\overline{\TG_\lambda} = \bigsqcup_{\mu \geq \lambda} \TG_\mu$, and by~\cite[Theorem~3.5]{mv} for $\lambda \in \XM$ and $\FC$ in $\Perv_{\GDO}(\Gr, \FM)$ we have
\begin{equation*}
\HF^n(\TG_\lambda, t_\lambda^! \FC)=0 \qquad \text{unless $i=\langle \lambda, 2{\check \rho} \rangle$,}
\end{equation*}
where $2{\check \rho} \in {\check \XM}$ is the sum of positive roots of $\GD$. From these observations one can deduce that the morphism
\[
\HF^{\langle \lambda, 2{\check \rho} \rangle}_{\TD}(\overline{\TG_\lambda}, \overline{t}_\lambda^! \FC) \to \HF^{\langle \lambda, 2{\check \rho} \rangle}_{\TD}(\TG_\lambda, t_\lambda^! \FC)
\]
induced by restriction, and the morphism
\[
\HF^{\langle \lambda, 2{\check \rho} \rangle}_{\TD}(\TG_\lambda, t_\lambda^! \FC) \to \HF^{\langle \lambda, 2{\check \rho} \rangle}(\TG_\lambda, t_\lambda^! \FC)
\]
induced by ``forgetting the $\TD$-equivariance,'' are both isomorphisms. Combining these identification and identifying $\HF^\bullet_{\TD}(\pt; \FM)$ with $\OC(\tg^*)$ as in~\S\ref{ss:notation-Satake} we obtain a natural morphism
\[
\OC(\tg^*) \otimes_\FM \HF^{\langle \lambda, 2{\check \rho} \rangle}(\TG_\lambda, t_\lambda^! \FC) \to \HF^{\bullet}_{\TD}(\overline{\TG_\lambda}, \overline{t}_\lambda^! \FC).
\]
Composing with the natural morphism $\HF^{\bullet}_{\TD}(\overline{\TG_\lambda}, \overline{t}_\lambda^! \FC) \to \HF^\bullet_{\TD}(\Gr, \FC)$ and then taking the direct sum for all $\lambda$, we obtain a canonical morphism
\begin{equation}
\label{eqn:cohom-Tequ-0}
\bigoplus_{\lambda \in \XM} \, \OC(\tg^*) \otimes_\FM \HF^{\langle \lambda, 2{\check \rho} \rangle}(\TG_\lambda, t_\lambda^! \FC) \to \HF^\bullet_{\TD}(\Gr, \FC).
\end{equation}
It is explained in~\cite[Lemma~2.2]{yz} that this morphism is an isomorphism. In fact, the image under $\FM \otimes_{\OC(\tg^*)} (-)$ of this morphism can be identified with the isomorphism
\[
\bigoplus_{\lambda \in \XM} \, \HF^{\langle \lambda, 2{\check \rho} \rangle}(\TG_\lambda, t_\lambda^! \FC) \simto \HF^\bullet(\Gr, \FC)
\]
of~\cite[Theorem~3.6]{mv}. Hence~\eqref{eqn:cohom-Tequ-0} can be interpreted as an isomorphism
\begin{equation}
\label{eqn:cohom-Tequ}
\OC(\tg^*) \otimes \FF(\FC) \simto \HF^\bullet_{\TD}(\Gr, \FC).
\end{equation}
By~\cite[Lemma~2.4]{yz} this isomorphism is compatible with the natural monoidal structures on both sides.

Since $\FC$ is $\GDO$-equivariant, the right-hand side in~\eqref{eqn:cohom-Tequ} is endowed with a natural action of $W$.

\begin{lem}
\label{lem:action-W-Gr}
Assume that $\FC$ is in $\PParity_{\GDO}(\Gr, \FM)$, and identify $\OC(\tg^*) \otimes \FF(\FC)$ with the space of (algebraic) functions on $\tg^*$ with values in the vector space $\FF(\FC)$.
For $\alpha \in \Delta$, under the isomorphism~\eqref{eqn:cohom-Tequ} the action of $s_\alpha$ is given by
\begin{equation}
\label{eqn:action-s-cohom}
(s_\alpha \cdot f)(\xi) = u_{-\alpha}(-|\alpha|^2 \langle \xi, h_\alpha \rangle) \cdot f(s_\alpha \xi).
\end{equation}
\end{lem}

\begin{proof}
First, let us assume that $\GD$ has semisimple rank $1$. Since any connected component of $\Gr$ is isomorphic (in a $\GDO$-equivariant way) to a connected component of the affine Grassmannian of the quotient $\GD / Z(\GD)^\circ$, we can assume that $\GD$ is semisimple, and then that $\GD = \mathrm{PGL}(2,\CM)$ (with $\BD$, resp.~$\TD$, the subgroup of lower-triangular, resp.~diagonal, matrices). Then $\XM=\ZM$ (where $\alpha$ corresponds to $2$), and $\Gr^1$ is the only $1$-dimensional $\GDO$-orbit on $\Gr$; this orbit is isomorphic to $\GD/\BD \cong \mathbb{P}^1$ via $g \BD \mapsto g \cdot L_1$.
Since our formula~\eqref{eqn:action-s-cohom} is preserved by tensor products and direct summands, and since any tilting $\mathrm{SL}(2,\FM)$-module is a direct summand of a tensor power of the natural representation $\FM^2$, it is sufficient to consider the case $\FC=\IC\CC(\Gr^1, \FM) = \underline{\DM}_{\Gr^1}[-1]$. 

This case will be checked by explicit computation.
We have
\[
\HF^\bullet_{\TD}(\Gr, \FC) = \HF_{\TD}^{\bullet-1}(\Gr^1, \underline{\DM}_{\Gr^1}) = \FM[x] \cdot [\{L_{1}\}] \oplus \FM[x] \cdot [\Gr^1],
\]
where we identify $\HF_{\TD}^{\bullet}(\mathrm{pt}; \FM)=\OC(\tg^*)$ with $\FM[x]$ in the standard way (so that $x$ corresponds to $h_\alpha$), and $[Z]$ is the fundamental class of $Z$ in equivariant Borel--Moore homology. In the identification~\eqref{eqn:cohom-Tequ} we have $\FF(\FC) = \FM \cdot [\{L_1\}] \oplus \FM \cdot [\Gr^1]$, where $[\{L_1\}]$ has weight $1$ and $[\Gr^1]$ has weight $-1$.
The action of the simple reflection $s_\alpha$ stabilizes $[\Gr^1]$, and sends $[\{L_1\}]$ to $[\{L_{-1}\}]= -x \cdot [\Gr^1] + [\{L_1\}]$.
It is well known (see e.g.~\cite[Proof of Proposition~5.6]{yz} or~\cite[Th{\'e}or{\`e}me~2]{baumann}) that $e$ sends $[\Gr^1]$ to $[\{L_1\}]$. Formula~\eqref{eqn:action-s-cohom} follows.

Now we consider the general case. Let $\LD$ be the Levi subgroup of $\GD$ containing $\TD$ whose roots are the coroots associated with $\alpha$ and $-\alpha$. Let also $\PD$ be the minimal parabolic subgroup containing $\BD$ with Levi factor $\LD$. Then we have a diagram
\[
\xymatrix{
\Gr & \Gr_{\PD} \ar[r] \ar[l] & \Gr_{\LD},
}
\]
where $\Gr_{\PD}$ and $\Gr_{\LD}$ are the affine Grassmannians of $\PD$ and $\LD$ respectively, and the associated restriction functor
\[
\RG^{\GD}_{\LD} \colon \Perv_{\GDO}(\Gr, \FM) \to \Perv_{\LDO}(\Gr_{\LD}, \FM),
\]
as e.g.~in~\cite[\S 6.3]{gr}. If $\LB$ is the Levi subgroup of $\GB$ whose roots are $\alpha$ and $-\alpha$, then we have a Satake equivalence $\Satake_{\LD}$ for $\LD$, with dual group $\LB$, and $\RG^{\GD}_{\LD}$ corresponds, under $\Satake$ and $\Satake_{\LD}$, to the restriction functor $\Rep(\GB) \to \Rep(\LB)$. It is also known that $\RG^{\GD}_{\LD}$ sends parity complexes to parity complexes, see~\cite[Theorem~1.6]{jmw2}.

Using isomorphism~\eqref{eqn:cohom-Tequ} this provides a canonical isomorphism
\[
\HF^\bullet_{\TD}(\Gr, \FC) \cong \HF^\bullet_{\TD}(\Gr, \RG^{\GD}_{\LD}(\FC)),
\]
which is easily seen to commute with the actions of $s_\alpha$ on both sides. If $e(\LD)$ denotes the canonical regular nilpotent element in the Lie algebra $\lg$ of $\LB$ provided by $\Satake_{\LD}$, then it follows from the proof of~\cite[Proposition~5.6]{yz} that $e_\alpha = |\alpha|^2 \cdot e(\LD)$. Hence $e_{-\alpha}$ is equal to $\frac{1}{|\alpha|^2}$ times the similar element in $\lg$, and we deduce our formula~\eqref{eqn:action-s-cohom} from its analogue for $\LD$ proved above.
\end{proof}

\subsection{Equivariant cohomology of $\Gr$}
\label{ss:equiv-coh}

%

In this subsection we assume further that $\GD$ is quasi-simple and simply connected.

Let $\theta \in \XM$ be the highest short root of $\GB$.
As in~\cite[Proposition~5.7]{yz}\footnote{In~\cite[Proposition~5.7]{yz} the bilinear form appears with a ``$-$'' sign. However this is a mistake; in fact in the definition of $d_V$ in Lemma~4.2 the coefficient $\frac{1}{2}$ should be corrected to $-\frac{1}{2}$ (as can be checked from the proof), and as a consequence the sign in~\cite[Proposition~5.7]{yz} should be changed from ``$-$'' to ``$+$.''} we consider the $W$-invariant $\QM$-valued bilinear form on $\XM:=X_*(\TD)$ defined by
\[
\langle \lambda, \mu \rangle_{\Gr} := 2 \cdot \frac{(\lambda, \mu)_\Kil}{(\theta, \theta)_\Kil} \qquad \text{where} \qquad (\nu, \eta)_\Kil := \sum_{{\check \alpha} \in {\check \Phi}} \langle \nu, {\check \alpha} \rangle \cdot \langle \eta, {\check \alpha} \rangle.
\]
Identifying $X_*(\TB)$ with the dual of $\XM$ in the standard way, this bilinear form defines a $W$-equivariant morphism $\tau \colon \XM \to \QM \otimes_\ZM X_*(\TB)$.

\begin{lem}
\label{lem:tau-roots}
For any $\alpha \in \Phi$ we have
\[
\tau(\alpha) = |\alpha|^2 \cdot \alpha^\vee,
\]
where $\alpha^\vee$ is the coroot associated with $\alpha$.
\end{lem}

\begin{proof}
Since $\tau$ is $W$-equivariant, we have $s_\alpha(\tau(\alpha))=-\tau(\alpha)$; therefore, $\tau(\alpha) = z_\alpha \alpha^\vee$ for some $z_\alpha \in \QM$. To compute $z_\alpha$, we remark that
\[
2 z_\alpha = \langle \tau(\alpha), \alpha \rangle = \langle \alpha, \alpha \rangle_\Gr = 2 \frac{(\alpha, \alpha)_\Kil}{(\theta, \theta)_\Kil} = 2 |\alpha|^2.
\]
Hence $z_\alpha = |\alpha|^2$, which completes the proof.
\end{proof}

Under our assumptions, the image of $|\alpha|^2$ in $\FM$ is invertible for any $\alpha$, and the images of simple roots (resp.~simple coroots) in $\tg^*$ (resp.~in $\tg$) form a basis of $\tg^*$ (resp.~$\tg$). Hence Lemma~\ref{lem:tau-roots} implies that $\tau$ induces an isomorphism from $\tg^* = \XM \otimes_\ZM \FM = \ZM \Phi \otimes_\ZM \FM$ to $\tg = X_*(\TB) \otimes_\ZM \FM = \ZM {\check \Phi} \otimes_\ZM \FM$. For $\alpha \in \Phi$, the same arguments as in Lemma~\ref{lem:tau-roots} show that
\begin{equation}
\label{eqn:tau-roots}
(d\alpha) \circ \tau = |\alpha|^2 \cdot h_\alpha \qquad \text{in $(\tg^*)^*=\tg$.}
\end{equation}

Now, recall the element $e \in \gg$ defined in~\S\ref{ss:tensor}. Using this element we can define $\Upsilon:= e+\bg \subset \gg$ as in~\S\ref{ss:principal-nilpotent}.
We also set $\Sigma:= e + \tg$, and consider the isomorphism $\tau_\Sigma \colon \tg^* \to \Sigma$ defined by $\tau_\Sigma (\xi) = e + \tau(\xi)$. We denote by $\IB_{\tg^*}$ the pullback under $\tau_\Sigma$ of the restriction $\IB_\Sigma$ of $\IB$ to $\Sigma$. Note that $\Sigma \subset \Upsilon \subset \gg_\reg$ (see~\eqref{eqn:Up-reg}), so that $\IB_{\tg^*}$ is a smooth commutative group scheme over $\tg^*$ by Corollary~\ref{cor:Ireg-smooth} and Corollary~\ref{cor:Ireg-commutative}.


\begin{prop}
\label{prop:cohom-Gr}
There exists a canonical $\FM$-algebra isomorphism
\[
\HF^\bullet_{\TD}(\Gr; \FM) \simto \mathrm{Dist}(\IB_{\tg^*}),
\]
where $\mathrm{Dist}$ denotes the distribution algebra.
\end{prop}

\begin{proof}
By~\cite[Theorem~6.1]{yz}, there exists a canonical isomorphism
\[
\HF^{\TD}_\bullet(\Gr; \FM) \simto \OC(\IB_{\tg^*}).
\]
Now by definition (see~\cite[\S 2.6]{yz}), $\HF^\bullet_{\TD}(\Gr; \FM)$ identifies with the graded dual of $\HM^{\TD}_\bullet(\Gr; \FM)$. Since the quotient of $\OC(\IB_{\tg^*})$ by any power of the augmentation ideal is finitely generated over $\OC(\tg^*)$, any distribution on $\IB_{\tg^*}$ belongs to this graded dual, so that we deduce an embedding
\[
\mathrm{Dist}(\IB_{\tg^*}) \hookrightarrow \HF^\bullet_{\TD}(\Gr; \FM).
\]
The image of this morphism under the functor $\FM \otimes_{\HF^\bullet_{\TD}(\mathrm{pt}; \FM)} (-)$ is an isomorphism by~\cite[Corollary~6.4]{yz}. (Here we use also~\cite[\S I.7.4(1)]{jantzen-gps} and the fact that smooth group schemes are infinitesimally flat, see~\cite[Proof of Lemma~5.15]{mr} for references.) By the graded Nakayama lemma, this implies that our morphism itself is an isomorphism.
\end{proof}

\subsection{Kostant--Whittaker reduction}
\label{ss:KW-reduction}

The results of this subsection apply in the general setting of \S\S\ref{ss:principal-nilpotent}--\ref{ss:centralizer-F}; in fact we only need to assume that the group $\GB$ satisfies condition (C3).

Let us consider the composition
\[
\sigma \colon \Coh^\GB(\gg_\reg) \to \Rep(\IB_\reg) \to \Rep(\IB_\Sigma)
\]
where as above $\Sigma=e+\tg$.
Here the first arrow is the functor considered in the proof of Proposition~\ref{prop:coh-greg}, and the second arrow is induced by restriction to $\Sigma$. 


Consider the action of $W$ on $\Sigma$ induced by the natural action on $\tg$. The restriction $\chi_\Sigma$ of  $\chi_\reg$ to $\Sigma$ is $W$-equivariant, where $W$ acts trivially on the codomain. It follows from Proposition~\ref{prop:J} that $\IB_\Sigma$ is canonically isomorphic to $\Sigma \times_{\tg/W} \JB$; therefore this group scheme is $W$-equivariant. We claim that $\kappa$ factors through a functor with values in the category $\Rep^W(\IB_\Sigma)$ of $W$-equivariant representations of $\IB_\Sigma$ which are coherent over $\OC_{\Sigma}$, i.e.~the category whose objects consist of a representation $\FC$ in $\Rep(\IB_\Sigma)$ together with a collection of isomorphisms $\varphi_w \colon w^*(\FC) \simto \FC$ compatible with composition (and identity) in the obvious sense. Indeed,
recall the equivalences
\[
\Coh^\GB(\gg_\reg) \simto \Rep^\UB(\IB_\Upsilon) \simto \Rep(\JB)
\]
from Proposition~\ref{prop:coh-greg} and Remark~\ref{rk:KW-reduction-canonical}, and denote by $\kappa$ this composition. 
Then we have $\sigma=\chi_\Sigma^* \circ \kappa$. Clearly $\chi_\Sigma^*$ factors through a functor $\Rep(\JB) \to \Rep^W(\IB_\Sigma)$, and our claim follows.

For $V$ in $\Rep(\GB)$, we set
\[
\varsigma(V) := \sigma(V \otimes_\FM \OC_{\gg_\reg}).
\]
As a coherent sheaf on $\Sigma$, $\varsigma(V)$ is isomorphic to $V \otimes \OC_\Sigma$. The $W$-action on the global sections of this coherent sheaf can be described more concretely as follows: since $\Upsilon$ is a $\UB$-torsor over $\tg/W$ (see Proposition~\ref{prop:ganginzburg}), there exists a canonical isomorphism
\[
\OC(\Upsilon) \otimes_{\OC(\tg/W)} \bigl( V \otimes \OC(\Upsilon) \bigr)^{\UB} \simto V \otimes \OC(\Upsilon).
\]
Restricting to $\Sigma$ we deduce an isomorphism
\[
\OC(\Sigma) \otimes_{\OC(\tg/W)} \bigl( V \otimes \OC(\Upsilon) \bigr)^{\UB} \simto V \otimes \OC(\Sigma) = \Gamma(\Sigma, \varsigma(V))
\]
Then the $W$-action is induced by the action on $\OC(\Sigma)$. 

\begin{lem}
\label{lem:morphisms-Coh}
For $V$, $V'$ in $\Rep(\GB)$, the functor $\sigma$ induces an isomorphism
\[
\Hom_{\Coh^\GB(\gg)}(V \otimes_\FM \OC_\gg, V' \otimes_\FM \OC_\gg) \simto \bigl( \Hom_{\Rep(\IB_\Sigma)}(\varsigma(V), \varsigma(V')) \bigr)^W.
\]
\end{lem}

\begin{proof}
First we recall that $\gg \smallsetminus \gg_\reg$ has codimension at least $2$ in $\gg$.\footnote{This fact is well known. For instance, the following proof works for any connected reductive group satisfying (C1). The same arguments as in~\cite[Proposition~1.9.4]{br} show that the complement of $\tgg_\reg$ in $\tgg$ has codimension at least $2$; then the case of $\gg$ follows, since $\gg \smallsetminus \gg_\reg$ is the image of $\tgg \smallsetminus \tgg_\reg$ in $\gg$ and $\dim(\tgg)=\dim(\gg)$.} Hence restriction induces an isomorphism
\[
\Hom_{\Coh^\GB(\gg)}(V \otimes_\FM \OC_\gg, V' \otimes_\FM \OC_\gg) \simto \Hom_{\Coh^\GB(\gg_\reg)}(V \otimes_\FM \OC_{\gg_\reg}, V' \otimes_\FM \OC_{\gg_\reg}).
\]
Now the equivalence $\kappa$ induces an isomorphism
\[
\Hom_{\Coh^\GB(\gg_\reg)}(V \otimes_\FM \OC_{\gg_\reg}, V' \otimes_\FM \OC_{\gg_\reg}) \simto \Hom_{\Rep(\JB)}(\varkappa(V), \varkappa(V')),
\]
where $\varkappa(V)=\kappa(V \otimes \OC_{\gg_{\reg}})$ and similarly for $V'$. And finally, using the fact that $\OC(\tg)$ is free over $\OC(\tg/W)$ one can easily check that the functor $\chi_\Sigma^*$ induces an isomorphism
\[
\Hom_{\Rep(\JB)}(\varkappa(V), \varkappa(V')) \simto \bigl( \Hom_{\Rep(\IB_\Sigma)}(\varsigma(V), \varsigma(V')) \bigr)^W,
\]
which finishes the proof.
\end{proof}


Below we will need an explicit description of the $W$-action on $\varsigma(V)$, as follows. Recall the isomorphism $u_{-\alpha} \colon \FM \to \UB_{-\alpha}$ defined in~\S\ref{ss:tensor}.

\begin{lem}
\label{lem:action-W-Coh}
Identifying the global sections on $\varsigma(V)$ with the space of (algebraic) functions from $\Sigma$ to $V$, the action of $s_\alpha$ satisfies
\[
(s_\alpha \cdot f)(e+h) = u_{-\alpha}(-(d\alpha)(h)) \cdot f(e+s_\alpha h) 
\]
for $h \in \tg$ and $f$ a global section of $\varsigma(V)$.
\end{lem}

\begin{proof}
By definition of the action, for $w \in W$ we have
\[
(w \cdot f)(x) = u(w,x) \cdot f(w^{-1} \cdot x),
\]
where $u(w,x) \in \UB$ is the unique element such that $u(w,x) \cdot (w^{-1} \cdot x) = x$. In our case we have $s_\alpha \cdot (e+h) = e+ s_\alpha h$, and 
a straightforward computation shows that
\[
u_{-\alpha}(-(d\alpha)(h)) \cdot (e+s_\alpha h) = e+ h.
\]
The formula follows.
\end{proof}

\subsection{Mixed modular derived Satake equivalence}
\label{ss:derived-Satake}

\emph{From now on we assume that $\GD$ is quasi-simple and simply connected.}

We define the \emph{mixed $\GDO$-equivariant derived category} of $\Gr$ with coefficients in $\FM$ by
\[
\Dmix_{\GDO}(\Gr, \FM) := \Kb \Parity_{\GDO}(\Gr, \FM);
\]
see~\S\ref{ss:intro-derived-Satake} for a justification of this definition. As in~\cite{modrap2} we denote by $\{1\}$ the auto-equivalence induced by the cohomological shift in the category $\Parity_{\GDO}(\Gr, \FM)$ and define the ``Tate twist'' by the formula $\langle 1 \rangle = \{-1\}[1]$, where $[1]$ is the cohomological shift in the triangulated category $\Kb \Parity_{\GDO}(\Gr, \FM)$. 
The restriction to $\Parity_{\GDO}(\Gr, \FM)$ of the convolution product $\star$ (see~\S\ref{ss:notation-Satake}) induces a convolution product on $\Dmix_{\GDO}(\Gr, \FM)$, which we again denote by $\star$.

On the other hand, we consider
the bounded derived category $\Db \Coh^{\GB \times \Gm}(\gg)$ of $\GB \times \Gm$-equivariant coherent sheaves on $\gg$, where $\GB \times \Gm$ acts on $\gg$ via $(g,t) \cdot \xi = t^{-2} (g \cdot \xi)$.
This category has ``shift functor'' $\langle 1 \rangle$, defined as the tensor product with the tautological $1$-dimensional $\Gm$-module. Since $\gg$ is a smooth variety, the category  $\Db \Coh^{\GB \times \Gm}(\gg)$ has a natural monoidal structure induced by (derived) tensor product of coherent sheaves. 

The main result of this section is the following. It provides a ``mixed modular'' analogue of the ``derived Satake equivalence'' of~\cite{bf}.

\begin{thm}
\label{thm:derived-Satake}
Assume that $\GD$ is quasi-simple and simply connected, and that $\ell$ is very good for $\GD$.
There exists an equivalence of monoidal triangulated categories
\[
\Phi \colon \Dmix_{\GDO}(\Gr, \FM) \simto \Db \Coh^{\GB \times \Gm}(\gg)
\]
which satisfies $\Phi \circ \langle 1 \rangle \cong \langle 1 \rangle [1] \circ \Phi$ and
\[
\Phi(\FC) \cong \Satake(\FC) \otimes_\FM \OC_\gg
\]
for all $\FC$ in $\PParity_{\GDO}(\Gr, \FM)$.
\end{thm}

The proof of this theorem is given in~\S\ref{ss:proof-derived-Satake}. In the remainder of this subsection we introduce some ``tilting'' objects in $\Db \Coh^{\GB \times \Gm}(\gg)$ from which one can recover the entire category. (We do not claim that these objects are tilting for any quasi-hereditary structure, but they will play the same role as the one played by actual tilting objects in~\cite{modrap2} or~\cite{mr}.)
More precisely, we denote by $\mathsf{T}^{\GB \times \Gm}(\gg)$ the additive, monoidal full subcategory of $\Coh^{\GB \times \Gm}(\gg)$ consisting of objects of the form $V \otimes \OC_\gg \langle n \rangle$ where $V \in \Tilt(\GB)$ and $n \in \ZM$.

\begin{lem}
\label{lem:Hom-vanishing-gg}
For $V$, $V'$ in $\Tilt(\GB)$ and for $n,m \in \ZM$ we have
\[
\Hom^i_{\Db \Coh^{\GB \times \Gm}(\gg)}(V \otimes \OC_\gg \langle n \rangle, V' \otimes \OC_\gg \langle m \rangle)=0 \qquad \text{for $i >0$.}
\]
\end{lem}

\begin{proof}
Since $\OC(\tg)$ is free over $\OC(\tg)^W$, see~\cite{demazure}, the object
\[
\Rder\pi_* \OC_{\tgg} \cong \OC_{\gg} \otimes_{\OC(\tg)^W} \OC(\tg)
\]
is a direct sum of $\Gm$-shifts of the structure sheaf $\OC_\gg$. (The isomorphism stated here is well known, see e.g.~\cite[Proof of Proposition~3.4.1]{bmr} for a proof under our assumptions.) Therefore, to prove the lemma it is enough to prove that
\[
\Hom^i_{\Db \Coh^{\GB \times \Gm}(\gg)}(V \otimes \OC_\gg \langle n \rangle, \Rder\pi_*(V' \otimes \OC_{\tgg}) \langle m \rangle)=0
\]
for all $i>0$. However, by adjunction we have
\begin{multline*}
\Hom^i_{\Db \Coh^{\GB \times \Gm}(\gg)}(V \otimes \OC_\gg \langle n \rangle, \Rder\pi_*(V' \otimes \OC_{\tgg}) \langle m \rangle) \cong \\
\Hom^i_{\Db \Coh^{\GB \times \Gm}(\tgg)}(V \otimes \OC_{\tgg} \langle n \rangle, V' \otimes \OC_{\tgg}\langle m \rangle).
\end{multline*}
Now the fact that the right-hand side vanishes follows from the similar claim on $\tNC$ (which itself follows from~\cite[Corollary~4.3.7]{mr-exotic}) by the same arguments as in~\cite[Proposition~5.5]{mr}.
\end{proof}

\begin{lem}
\label{lem:DCoh-generated-free}
The category $\Db \Coh^{\GB \times \Gm}(\gg)$ is generated, as a triangulated category, by the ``free'' objects of the form $V \otimes \OC_{\gg} \langle i \rangle$ for $V$ in $\Rep(\GB)$ and $i \in \ZM$.
\end{lem}

\begin{proof}
The proof of~\cite[Lemma~5.7]{achar} applies in our situation; let us recall its main steps. First, it is clear that any $\GB \times \Gm$-equivariant coherent sheaf on $\gg$ is a quotient of a direct sum of free objects. Then, using an appropriate form of Hilbert's syzygy theorem, to conclude it suffices to prove that any $\GB \times \Gm$-equivariant $\OC(\gg)$-module $M$ which is graded free as an $\OC(\gg)$-module admits a filtration, as a $\GB \times \Gm$-equivariant module, by free objects. However, if $N$ is the lowest degree in which $M$ is non-zero, the subspace $M_N \subset M$ of elements of degree $N$ is $\GB \times \Gm$-stable, the natural morphism $\OC(\gg) \otimes M_N \to M$ is an injective morphism of $\GB \times \Gm$-equivariant $\OC(\gg)$-modules, and its cokernel is graded-free over $\OC(\gg)$, of rank smaller than $M$. We conclude by induction.
\end{proof}

\begin{cor}
\label{cor:DbCoh-tilting}
There exists an equivalence of monoidal triangulated categories
\[
\Kb \mathsf{T}^{\GB \times \Gm}(\gg) \simto \Db \Coh^{\GB \times \Gm}(\gg).
\]
\end{cor}

\begin{proof}
Consider the composition
\[
\Kb \mathsf{T}^{\GB \times \Gm}(\gg) \hookrightarrow \Kb \Coh^{\GB \times \Gm}(\gg) \to \Db \Coh^{\GB \times \Gm}(\gg).
\]
Lemma~\ref{lem:Hom-vanishing-gg} (together with standard arguments) implies that this functor is fully-faithful. Since the (finite-dimensional) tilting $\GB$-modules generate the derived category $\Db \Rep(\GB)$, using Lemma~\ref{lem:DCoh-generated-free} we obtain that it is also essentially surjective. Finally, our functor is clearly monoidal, hence the corollary is proved.
\end{proof}

\subsection{Proof of Theorem~\ref{thm:derived-Satake}}
\label{ss:proof-derived-Satake}

By Corollary~\ref{cor:DbCoh-tilting} and the definition of the category $\Dmix_{\GDO}(\Gr, \FM)$, to prove the theorem it suffices to construct an equivalence of monoidal additive categories
\begin{equation}
\label{eqn:equiv-parity-tilt}
\Parity_{\GDO}(\Gr, \FM) \simto \mathsf{T}^{\GB \times \Gm}(\gg)
\end{equation}
which intertwines the shift functors $\{1\}$ and $\langle -1 \rangle$ and sends any object $\FC$ in $\PParity_{\GDO}(\Gr, \FM)$ to $\Satake(\FC) \otimes_\FM \OC_\gg$.
Now $\Parity_{\GDO}(\Gr, \FM)$ is equivalent, as an additive monoidal category, to the category whose objects are sequences $(\FC_i)_{i \in \ZM}$ of objects of $\PParity_{\GDO}(\Gr, \FM)$ satisfying $\FC_i=0$ for all but finitely many $i$'s (where morphisms are defined in the obvious way, and with the obvious ``graded'' convolution product), via the functor
\[
(\FC_i)_{i \in \ZM} \mapsto \bigoplus_{i \in \ZM} \FC_i [-i].
\]
Therefore, to construct an equivalence as in~\eqref{eqn:equiv-parity-tilt} it is enough to prove that the assignment $\FC \mapsto \Psi(\FC):=\Satake(\FC) \otimes_\FM \OC_{\gg}$ (which is obviously compatible with the monoidal structures) can be extended to a monoidal functor; in other words
we have to define, for any $i \in \ZM$ and $\FC$, $\GC$ in $\PParity_{\GDO}(\Gr, \FM)$, an isomorphism
\begin{equation}
\label{eqn:def-functor-morphisms}
\Hom_{\Parity_{\GDO}(\Gr, \FM)}(\FC, \GC \{i\}) \simto \Hom_{\mathsf{T}^{\GB \times \Gm}}(\Psi(\FC), \Psi(\GC) \langle -i \rangle),
\end{equation}
compatible with all our structures (and in particular with composition).
To define this isomorphism we will describe both sides in similar terms. (This strategy is of course reminiscent of the ``Soergel construction'' used in particular in~\cite{bf, modrap2, mr}).

First we consider the left-hand side of~\eqref{eqn:def-functor-morphisms}.
We claim that there exists a canonical isomorphism
\[
\Hom_{\Parity_{\GDO}(\Gr, \FM)}(\FC, \GC \{i\}) \simto \bigl( \Hom_{\Db_{\TD}(\Gr, \FM)}(\FC, \GC \{i\}) \bigr)^W,
\]
where $\Db_{\TD}(\Gr, \FM)$ is the $\TD$-equivariant constructible derived category of $\Gr$, and $W$ acts naturally on $\Hom_{\Db_{\TD}(\Gr, \FM)}(\FC, \GC \{i\})$. In fact, since we have $\HF^\bullet_{\GD}(\mathrm{pt}; \FM) \cong (\HF^\bullet_{\TD}(\mathrm{pt}; \FM))^W$ (see~\S\ref{ss:notation-Satake}), to prove this it suffices to prove that the $\HF^\bullet_{\GD}(\mathrm{pt}; \FM)$-module $\Hom^\bullet_{\Db_{\GDO}(\Gr, \FM)}(\FC, \GC)$ is free and that the natural morphism
\[
\HF^\bullet_{\TD}(\mathrm{pt}; \FM) \otimes_{\HF^\bullet_{\GD}(\mathrm{pt}; \FM)} \Hom^\bullet_{\Db_{\GDO}(\Gr, \FM)}(\FC, \GC) \to \Hom^\bullet_{\Db_{\TD}(\Gr, \FM)}(\FC, \GC)
\]
is an isomorphism. Now, using~\cite[Proposition~2.6]{jmw}, to prove this fact it suffices to prove that for any $\lambda \in \XM$ the $\HF^\bullet_{\GD}(\mathrm{pt}; \FM)$-module $\HF^\bullet_{\GD}(\Gr^\lambda; \FM)$ is free, and that the natural morphism
\[
\HF^\bullet_{\TD}(\mathrm{pt}; \FM) \otimes_{\HF^\bullet_{\GD}(\mathrm{pt}; \FM)} \HF^\bullet_{\GD}(\Gr^\lambda; \FM) \to \HF^\bullet_{\TD}(\Gr^\lambda; \FM)
\]
is an isomorphism. Finally, since $\Gr^\lambda$ is an affine bundle over a partial flag variety of $\GD$, it is enough to prove the similar claim for each partial flag variety $\GD/\PD$. This result is well known, and can e.g.~be proved by a spectral sequence argument. 

Now, recall that the functor $\HF^\bullet_{\TD}(\Gr, -)$ induces an isomorphism
\[
\Hom^\bullet_{\Db_{\TD}(\Gr, \FM)}(\FC, \GC) \simto \Hom_{\HF^\bullet_{\TD}(\Gr; \FM)}(\HF_{\TD}^\bullet(\Gr, \FC), \HF_{\TD}^\bullet(\Gr, \GC)),
\]
see~\cite[Proposition~3.12]{mr}.\footnote{In~\cite{mr} we work with integral coefficients and with a particular class of parity complexes. However the same arguments work in the case of fields, and in this setting they apply to all parity complexes; see~\cite[Remark~3.18]{mr}.} 
We deduce a canonical isomorphism
\begin{equation}
\label{eqn:proof-derived-Satake-Hom-Gr}
\Hom^\bullet_{\Parity_{\GDO}(\Gr, \FM)}(\FC, \GC) \simto \bigl( \Hom_{\HF^\bullet_{\TD}(\Gr; \FM)}(\HF_{\TD}^\bullet(\Gr, \FC), \HF_{\TD}^\bullet(\Gr, \GC)) \bigr)^W.
\end{equation}


Now we consider the right-hand side of~\eqref{eqn:def-functor-morphisms}. First we observe that the forgetful functor induces an isomorphism
\begin{multline*}
\bigoplus_{i \in \ZM} \Hom_{\Coh^{\GB \times \Gm}(\gg)}(\Satake(\FC) \otimes \OC_{\gg}, \Satake(\GC) \otimes \OC_{\gg} \langle -i \rangle) \\
\simto \Hom_{\Coh^{\GB}(\gg)}(\Satake(\FC) \otimes \OC_{\gg}, \Satake(\GC) \otimes \OC_{\gg}).
\end{multline*}
Next, by Lemma~\ref{lem:morphisms-Coh} we have a canonical isomorphism
\[
\Hom_{\Coh^{\GB}(\gg)}(\Satake(\FC) \otimes \OC_{\gg}, \Satake(\GC) \otimes \OC_{\gg}) \simto \bigl( \Hom_{\Rep(\IB_\Sigma)}(\varsigma(\Satake(\FC)), \varsigma(\Satake(\GC))) \bigr)^W.
\]
Using~\cite[Lemma~I.7.16]{jantzen-gps} we deduce a canonical isomorphism
\[
\Hom_{\Coh^{\GB}(\gg)}(\Psi(\FC), \Psi(\GC)) \simto \bigl( \Hom_{\mathrm{Dist}(\IB_\Sigma)}(\varsigma(\Satake(\FC)), \varsigma(\Satake(\GC))) \bigr)^W.
\]
(See the proof of Proposition~\ref{prop:cohom-Gr} for remarks on the infinitesimal flatness assumption.)
Finally, since $\tau_\Sigma$ is an isomorphism we deduce a canonical isomorphism
\begin{equation}
\label{eqn:proof-derived-Satake-Hom-Coh}
\Hom_{\Coh^{\GB}(\gg)}(\Psi(\FC), \Psi(\GC)) \simto \bigl( \Hom_{\mathrm{Dist}(\IB_{\tg^*})}(\tau_\Sigma^* \varsigma(\Satake(\FC)), \tau_\Sigma^* \varsigma(\Satake(\GC))) \bigr)^W.
\end{equation}

By construction (see~\cite{yz}), isomorphism~\eqref{eqn:cohom-Tequ} is $\HF^\bullet_{\TD}(\Gr; \FM)$-equivariant, where $\HF^\bullet_{\TD}(\Gr; \FM)$ acts on the left-hand side via the isomorphism of Proposition~\ref{prop:cohom-Gr}. 
With the present notation, this isomorphism can therefore be written as an isomorphism of $\HF^\bullet_{\TD}(\Gr; \FM)$-modules
\[
\HF_{\TD}^\bullet(\Gr, \FC) \simto \tau_\Sigma^* \bigl( \varsigma(\Satake(\FC)) \bigr).
\]
Comparing Lemma~\ref{lem:action-W-Gr} with Lemma~\ref{lem:action-W-Coh}, and using~\eqref{eqn:tau-roots}, we observe that this isomorphism is $W$-equivariant. It is also compatible with the natural monoidal structure on both sides. Hence, using this identification and comparing~\eqref{eqn:proof-derived-Satake-Hom-Gr} with~\eqref{eqn:proof-derived-Satake-Hom-Coh} we deduce an isomorphism as in~\eqref{eqn:def-functor-morphisms}, with all the necessary compatibility properties.

\end{document}